\documentclass [final]{amsart}

\usepackage{amsmath,amsfonts,amssymb,latexsym,hyperref}
\usepackage[foot]{amsaddr}
\usepackage{showkeys}
\usepackage{float} 
\usepackage{graphicx,cite} 

\usepackage{algorithm} 
\usepackage{algpseudocode}
\usepackage{todonotes}

\newtheorem{theorem}{Theorem}

\newtheorem{definition}[theorem]{Definition}
\theoremstyle{remark}
\newtheorem{remark}[theorem]{Remark}

\theoremstyle{definition}

\usepackage{color}

\newcommand{\R}{\mathbb{R}}

\numberwithin{theorem}{section}

\allowdisplaybreaks

\author[J. Y. Lee]{Jae Yong Lee$^\dagger$}
\address{$^\dagger$Department of Mathematics, Pohang University of Science and Technology (POSTECH), Pohang 37673, Republic of Korea.
}

\author[J. W. Jang]{Jin Woo Jang$^\ddagger$}
\address{$^\ddagger$Institute for Applied Mathematics,
University of Bonn, Endenicher Allee 60, 53115 Bonn, Germany
}

\author[H. J. Hwang]{Hyung Ju Hwang$^\dagger$}

 \AtBeginDocument{%
 	\def\MR#1{}
 }
 \let\svthefootnote\thefootnote
\begin{document}

	\newcommand{\eqdef }{\overset{\mbox{\tiny{def}}}{=}}
	\newcommand{\rth}{{\mathbb{R}^3}}
	\thispagestyle{empty}
	\allowdisplaybreaks
	\let\thefootnote\relax\footnotetext{2010 \textit{Mathematics Subject Classification.} Primary: 68T20, 35Q84, 35B40, 82C40.\\
		\textit{Key words and phrases.}  Vlasov-Poisson-Fokker-Planck system, Poisson-Nernst-Planck system, diffusion limit, artificial neural network, asymptotic-preserving scheme.
		\email{\href{mailto:hjhwang@postech.ac.kr}{hjhwang@postech.ac.kr}} }
		\addtocounter{footnote}{-1}\let\thefootnote\svthefootnote
\title[]{The model reduction of the Vlasov-Poisson-Fokker-Planck system to the Poisson-Nernst-Planck system via the Deep Neural Network Approach}

\begin{abstract}The model reduction of a mesoscopic kinetic dynamics to a macroscopic continuum dynamics has been one of the fundamental questions in mathematical physics since Hilbert's time. In this paper, we consider a diagram of the diffusion limit from the Vlasov-Poisson-Fokker-Planck (VPFP) system on a bounded interval with the specular reflection boundary condition to the Poisson-Nernst–Planck (PNP) system with the no-flux boundary condition. We provide a Deep Learning algorithm to simulate the VPFP system and the PNP system by computing the time-asymptotic behaviors of the solution and the physical quantities. We analyze the convergence of the neural network solution of the VPFP system to that of the PNP system via the Asymptotic-Preserving (AP) scheme. Also, we provide several theoretical evidence that the Deep Neural Network (DNN) solutions to the VPFP and the PNP systems converge to the \textit{a priori} classical solutions of each system if the total loss function vanishes.
\end{abstract}
\maketitle

\tableofcontents 

\section{Introduction}
\subsection{Motivation: a diagram of diffusion limit}The description of physical dynamics in various scales is one of the main questions of interest in the mathematical modeling of complex systems. In kinetic theory, the description of the evolution of gases has been explained via the statistical approach on the probabilistic distribution functions on the mesoscopic level, whereas the fluid theory describes the dynamics on the macroscopic level. Each of these interpretations and the asymptotic expansions of the mesoscopic equations to the macroscopic equations have been crucial issues.

The aim of this paper is to establish the commutation of the following diagram of diffusion limit, which provides the reduction of the kinetic equation (the Vlasov-Poisson-Fokker-Planck system) to the fluid equation (the Poisson-Nernst–Planck system) as the perturbation parameter $\varepsilon$ tends to zero:
\begin{figure}[H]
  \includegraphics[width=\textwidth, draft=false]{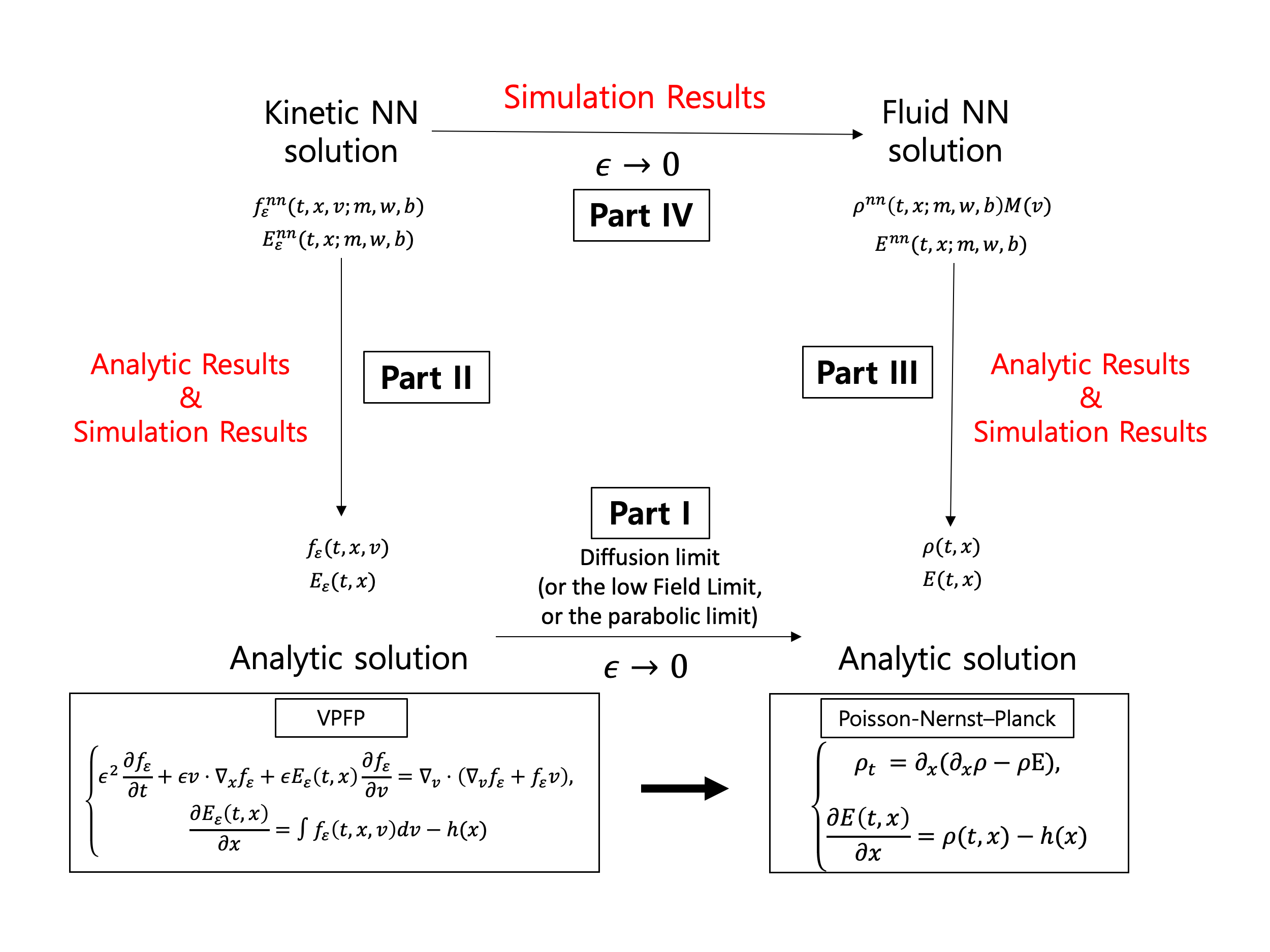}
  \caption{The diagram of diffusion limit}
  \label{fig:diagram}
\end{figure}We refer to a theoretical result from \cite{MR3294407} to obtain the bottom side (Part I) of the diagram. For the left-hand side (Part II), the right-hand side (Part III), and the upper side (Part IV) of the diagram, we use a Deep Learning method using the Deep Neural Network (DNN) solution to approximate the solutions to the kinetic equation and the fluid equation.
We provide large time behaviors and the steady-states of several physical moments of these DNN solutions to show an agreement with the theoretical results. Also, we provide theoretical evidence on the relationship between DNN solutions and the analytic solutions for the left and right-hand sides of the diagram.

There are many numerical studies to simulate an initial-boundary value problem for the kinetic and fluid equations. Especially, it is computationally challenging in a numerical scheme to automatically capture the limit for the asymptotic expansions on the small parameter (e.g. the parameter $\varepsilon$ tends to zero as the upper side of Figure \ref{fig:diagram}). Many numerical schemes have been developed to overcome this challenge. These schemes are the so-called Asymptotic-Preserving (AP) schemes, which have been firstly introduced by Jin \cite{MR1718639}. The key idea is to develop a numerical scheme to preserve the asymptotic limits from a mesoscopic to a macroscopic model in the fixed discrete setting.

A Deep Learning method has achieved remarkable success in various areas. Many studies have recently been introduced for learning partial differential equations (PDEs) using the Deep Learning method. These studies approximate the solutions of PDEs using a neural network architecture as a function approximator based on a universal approximation theorem \cite{MR1015670}. Along with many numerical methods, this Deep Learning approach has been proposed as a new way to simulate PDE problems.

In this paper, we provide a Deep Learning algorithm to simulate each side of Figure~\ref{fig:diagram}. In addition, we prove that the Deep Neural Network solutions converge to the analytic solutions. The simulation results of the Deep Learning support the model reduction of the Vlasov-Poisson-Fokker-Planck solution to the Poisson-Nernst–Planck solution.

\subsection{Main results, difficulties and our strategy}
In this paper, we establish the commutation of the diagram in Figure \ref{fig:diagram} on the diffusion limit from the Vlasov-Poisson-Fokker-Planck (VPFP) system \eqref{VPFP} in a bounded interval with the specular reflection boundary condition \eqref{specular} to the Poisson-Nernst–Planck (PNP) system \eqref{PNP} with the no-flux boundary condition \eqref{no-flux}. We provide a Deep Learning Algorithm \ref{algorithm} to simulate the VPFP system and the PNP system. The Deep Learning is a new possible approach for solving partial differential equations (PDEs) with many advantages. The Deep Learning method is a mesh-free approximation method, while the traditional methods in numerical analysis such as the Finite Difference schemes are influenced by the mesh. In this work, we randomly sample the grid points from a given domain in every epoch to train the Deep Neural Network. By randomly sampling each domain, we avoid the need of generating a mesh. We use the PyTorch library and the Adam optimizer for our Deep Learing algorithm.

One of the main advantages of the Deep Learning is that the algorithm can be implemented easily and intuitively by the computation of the gradient of the loss function with the chain rule, so-called the back-propagation algorithm. Namely, this means that it is easy to put the PDE information into the algorithm by adding terms to the loss function. In this regard, we propose the specific loss functions suitable for the VPFP system \eqref{VPFP_loss_total} in Section \ref{subsec:part2_loss} and for the PNP system \eqref{pnp_loss_total} in Section \ref{subsec:part3_loss}. Also, we use two neural network architectures at the same time to solve the Vlasov-Fokker-Planck system as in Figure \ref{fig:DNN_structure_VPFP} and the Nernst–Planck system as in Figure \ref{fig:DNN_structure_PNP}, which are coupled with the Poisson equation. In this way, there is a considerable advantage compared to traditional numerical methods, for which we needed to find a numerical scheme for each equation. We also use the Softplus activation function ($\bar{\sigma}(x)=\ln (1+e^x)$) for the output layer in the neural network structure. Since the Softplus function has the output in the scale of $(0,+\infty)$, it makes the neural network output $f_\varepsilon^{nn}(t,x,v;m,w,b)$ positive.

Meanwhile, there are some weaknesses of the Deep Learning approach. First of all, it is difficult to evaluate the accuracy of Deep Learning while the numerical methods have well-known error bounds. Also, it is difficult to show that the Deep Learning algorithms find the exact global minimum in the optimization aspect. Since the Deep Learning is the gradient-based approach, we cannot guarantee that we could find the global minimum of the loss function, even in our case using the \textit{Adam} optimizer. Instead, in order to show that our DNN solutions well-approximate the solutions of the VPFP system and the solutions of the PNP system, we provide numerical simulations to demonstrate that the neural network solutions satisfy the given theoretical predictions on the asymptotic behaviors of the VPFP system in Section \ref{subsec:part2_simulation} and of the PNP system in Section \ref{subsec:part3_simulation}. We analyze our DNN solutions via computing the steady-states for the solutions and via computing the physical quantities of the total mass, the kinetic energy, the entropy, the electric energy and the free energy, and their steady-states.

We also provide the theoretical supports for the convergence of the DNN solution to the \text{a priori} classical solution in two theorems. More precisely, for the VPFP system in Part II (Section \ref{subsec:part2_theory}), we claim in Theorem \ref{thm:vpfp_loss_0} that we can find a sequence of approximated neural network solutions that reduce the loss function \eqref{VPFP_loss_total}. Also, we prove an additional theorem (Theorem \ref{thm:vpfp_converge}) which states that the neural network solution converges to an analytic solution as we minimize the loss function \eqref{VPFP_loss_total}. In the proof, we use the transformation for the Vlasov-Poisson-Fokker-Planck system motivated by Carrillo \cite{MR1634851}. Similarly, we provide Theorem \ref{thm:pnp_converge} for the PNP system in Part III (Section \ref{subsec:part3_theory}).

In particular, it is hard to capture the asymptotic limit using the fixed numerical discretization in numerical schemes. In this work, we propose a newly devised technique, `Grid Reuse' method, which boosts the Deep Learning by adding the grid points that have the largest values of the loss function in every epoch. The `Grid Reuse' method makes it possible to approximate the neural network solution of the VPFP system with a small Knudsen number $\varepsilon$ without changing the number of grids sampled. We provide the numerical simulation for the trend of the diffusion limit from the DNN solution of the VPFP system to the DNN solution of the PNP system in Part IV (Section \ref{sec:part4}). To the best of authors' knowledge, this is the first attempt to use the Deep Learning method as an AP scheme to see the trend of the diffusion limit of the Vlasov-Poisson-Fokker-Planck system.

The main distictive of this paper compared to the numerical methods is the use of the neural network approach as a function approximator for the VPFP system, the PNP system, and the AP scheme as the Knudsen number $\varepsilon$ goes to 0. In this paper, our main goal is to complete the commutation in Figure \ref{fig:diagram} of the neural network version similar to Figure \ref{fig:AP} of the numerical analysis version.

\subsection{The Vlasov-Poisson-Fokker-Planck equation}In order to study the diffusion limit of the Vlasov-Poisson-Fokker-Planck system in a bounded interval $\Omega\eqdef(-1,1)$, we need to rescale the VPFP system with the Knudsen number $\varepsilon$. The small parameter $\varepsilon$ represents the ratio of the mean free path of the particles to the typical macroscopic length scale of the particle flow. We are interested in the scaling of the system using the change of variables $t'=\varepsilon^2t$ and $x'=\varepsilon x$; see Section 4 of \cite{MR1726024} and Section 1 of \cite{MR3294407}. With these variables, the VPFP system in a bounded interval $\Omega=(-1,1)$ can be written in the dimensionless form as follows:
\begin{equation}\label{VPFP}\begin{split}
\varepsilon^2\partial_t f_\varepsilon+\varepsilon v\partial_xf_\varepsilon+\varepsilon E_\varepsilon\partial_v f_\varepsilon &=\partial_v (vf_\varepsilon+ \partial_v f_\varepsilon),\hspace{3mm}t\in [0,T], \hspace{1mm}x\in\Omega, \hspace{1mm}v\in {{\mathbb{R}}},\\f_\varepsilon(0,x,v) &=f_0(x,v),
\\\partial_x E_\varepsilon&=\int_{{\mathbb{R}}}f_\varepsilon dv-h(x),\ x\in \Omega , \\
E_\varepsilon(0,-1)&=0,\\  E_\varepsilon(t,x)&=0,\ x\in\partial\Omega,
\end{split}
\end{equation}
where $f_\varepsilon =f_\varepsilon (t,x,v)\geq 0$ is the distribution of particles in $(t,x,v)\in(0,T)\times \Omega \times {{\mathbb{R}}}$, and $E_\varepsilon=E_\varepsilon(t,x)$ is the self-consistent electric force, where $\Phi_\varepsilon (t,x)$ is the internal potential of the system with the equation
$$E_\varepsilon (t,x)=-\partial_x\Phi_\varepsilon (t,x).$$
The function $h(x)$, which is on the right-hand side of the Poisson equation, stands for the presence of a background charge (e.g. ions). We assume the global neutrality condition 
\begin{equation}\label{VPFP_neutrality}
	\int_\Omega\int_{{\mathbb{R}}}f_\varepsilon dvdx-\int_\Omega h(x)dx=0.
\end{equation}
Note that the equations \eqref{VPFP}$_2$,\eqref{VPFP}$_3$, and \eqref{VPFP}$_4$ will together imply
\begin{equation}\label{VPFP_E_ini}
  0=E_\varepsilon(0,x)=\int_{-1}^x \partial_x E_\varepsilon(0,y)dy = \int_{-1}^x \int_{\mathbb{R}}f_0(y,v)dvdy-\int_{-1}^x h(y) dy.
\end{equation}

The existence and the uniqueness of the VPFP system have been well-studied. Victory and O'Dwyer in \cite{MR1052014} showed the existence of the classical solutions to the VPFP system in two dimension. Rein and Weckler in \cite{MR1178396} and Bouchut in \cite{MR1200643} showed the existence of global solutions to the three-dimensional VPFP system in the whole space. We refer to \cite{MR1634851} for the global weak solutions of VPFP system in a bounded domain with absorbing and reflection type boundary conditions. The large time asymptotic solutions to the Vlasov-Fokker-Planck equation has been studied first in \cite{MR897264,MR1115290} in the case that the particles occupy the whole space. They prove that the distribution function $f(t,x,v)$ tends to a Maxwellian function. This result has been extended by F. Bouchut and J. Dolbeault in \cite{MR1306570} under the more general assumption on the external potential. Also, we refer to \cite{MR1414375,MR1677677} in whole space domain.
In the case of initial-boundary value problem for the VPFP system, Bonilla et al. \cite{MR1470927} studied the large time asymptotic behaviors of the solutions with the reflection type boundary condition. The authors in \cite{MR4076068} considered the global well-posedness of the nonlinear Fokker-Planck equation with the specular reflection boundary. Also, the well-posedness and regularity for different boundary conditions to the kinetic Fokker-Planck equation has been studied in \cite{MR3007749,MR3436235,MR3788197,MR3237885,MR3902464,MR3897919,MR3906998,MR3614499}.

\subsection{The Poisson-Nernst–Planck equation}One of the macroscopic models to describe the distribution and the transport of ionic species is the Poisson-Nernst-Planck (PNP) system, where it is also often called the Drift-Diffusion-Poisson (DDP) equation. The PNP system consists of the Nernst-Planck equation that describes the drift and diffusion of ion and the Poisson equation that describes the effect of the self-consistent electric field. In this paper, we consider the following the 1-dimensional Poisson-Nernst-Planck (PNP) system in a bounded interval $\Omega=(-1,1)$: \begin{equation}\label{PNP}\begin{split}
\partial_t \rho &= \partial_x(\partial_x\rho-\rho E),\hspace{3mm}t\in [0,T], \hspace{1mm}x\in\Omega,
\\\rho(0,x) &=\rho_0(x),
\\\partial_x E&=\rho(t,x)-h(x),\ x\in \Omega,
\\E(0,-1)&=0,
\\ E(t,x)&=0,\ x\in\partial\Omega.
\end{split}
\end{equation}
Here, $\rho=\rho(t,x)$ stands for density of particles and $E(t,x)$ is the self-consistent electric force with the relation $E(t,x)=-\partial_x\Phi(t,x)$, similarly to the VPFP system. Then \eqref{PNP}$_2$,\eqref{PNP}$_3$ and \eqref{PNP}$_4$ will together imply 
\begin{equation}\label{PNP_E_ini}
  0=E(0,x)=\int_{-1}^x \partial_x E(0,y)dy = \int_{-1}^x\rho_0(y)dy-\int_{-1}^x h(y)dy.
\end{equation}
We also assume the neutrality condition for the background charge $h(x)$ as follows:
\begin{equation}\label{PNP_neutrality}
	\int_\Omega\rho(t,x)dx-\int_\Omega h(x)dx=0.
\end{equation}

The PNP system has a number of applications in many fields, such as electrical engineering, elctrokinetics, elctrochemistry and biophysics. Therfore, the analytical study of the PNP system has also a long history in the various context. An initial-boundary value problem for a system on the transport of mobile carriers in a semiconductor is studied by Gajewski and Groger in \cite{MR826656}. The existence and large time behavior of the PNP equation is studied in \cite{MR1305769}. Also, the convergence rate of solutions to the PNP system is studied in \cite{MR1770444,biler2000long}. We refer to a review paper \cite{MR3023368} for a recent development of generalized PNP systems.

\subsection{Boundary conditions}
\subsubsection{Phase boundary and the specular reflection for the VPFP system}\label{subsec:boundary_VPFP}
Throughout this paper, we will denote the phase boundary in $\Omega\times {{\mathbb{R}}}$ as $\gamma\eqdef \partial\Omega\times {{\mathbb{R}}}.$ Additionally we split this boundary into an outgoing boundary $\gamma_+$, an incoming boundary $\gamma_-$, and a singular boundary $\gamma_0$ for grazing velocities, defined as
\begin{equation}
\begin{split}
\gamma_+&\eqdef \{(x,v)\in\partial\Omega\times {{\mathbb{R}}} :n_x\cdot v >0\},\\
\gamma_-&\eqdef \{(x,v)\in\partial\Omega\times {{\mathbb{R}}} :n_x\cdot v <0\},\\
\gamma_0&\eqdef \{(x,v)\in\partial\Omega\times {{\mathbb{R}}} :n_x\cdot v =0\},
\end{split}
\end{equation} where $n_x$ is the outward normal vector.
We define the boundary integration for $f(x,v), \ (x,v)\in \partial \Omega \times {{\mathbb{R}}}$, $$\int_{\gamma_\pm}fd\gamma=\int_{\gamma_\pm}f(x,v)|n_x\cdot v |dS_xdv,$$ where $dS_x$ is the standard surface measure on $\partial\Omega$ and denote $$\int_{\gamma}fd\gamma=\int_{\gamma_+}f(x,v)d\gamma-\int_{\gamma_-}f(x,v)d\gamma.$$ We also define the $L^2(\gamma)$ norm with respect to the measure $|n_x\cdot v |dS_xdv$, $$||f||^2_\gamma\eqdef ||f||^2_{\gamma_+}+||f||^2_{\gamma_-}.$$

In terms of $f$, we formulate the specular reflection boundary condition as \begin{equation}\label{specular}
	f(t,x,v)|_{\gamma_-}=f(t,x,R(x)v),
\end{equation} for all $x\in\partial \Omega$,
where $$R(x)v\eqdef v-2n_x(n_x\cdot v).$$

One of the well-known a priori conservation laws for the Vlasov-Poisson-Fokker-Planck system (\ref{VPFP}) for the specular boundary conditions is the conservation of mass. It means that the total mass, ``\text{Mass}", of distribution $f_\varepsilon(t,x,v)$ is preserved for any time as follows:
\begin{equation}\label{mass conservation}
  \text{Mass}(t)\eqdef\int_\Omega\int_\R f_\varepsilon(t,x,v) dv dx \equiv\int_\Omega\int_\R f_0(x,v),
\end{equation}
which means $\frac{d}{dt}\text{Mass}(t)=0$.

\subsubsection{No-flux Boundary Condition for the PNP system}\label{subsec:boundary_PNP}The PNP system is usually posed in a bounded domain with some boundary condition. In this paper, we use the no-flux boundary condition for the PNP equation as follows:
\begin{equation}\label{no-flux}
	(\partial_x\rho-\rho E)\cdot n_x=0,\ x\in\partial\Omega.
\end{equation}
This condition is one of the natural boundary conditions for a macroscopic model to explain the diffusion of ions under the effect of potential. With the Dirichlet boundary condition \eqref{PNP}$_5$ for the Poisson equation, it reduces the following boundary conditions:
$$\partial_x\rho(t,x)=E(t,x)=0,\ x\in\partial\Omega.$$

This boundary condition implies that the system (\ref{VPFP}) has the conservation of charges/ions. We can check this property by integrating \eqref{PNP}$_1$ with respect to $x$ over the whole domain $[-1,1]$ as follows:
$$\partial_t \left(\int_\Omega\rho dx\right)=\left[\partial_{x}\rho-\rho E\right]^{1}_{t=-1}=0.$$
This implies that the convservation of total density, that is,
\begin{equation}\label{mass rho conservation}
  \text{Mass$_\rho$}(t)\eqdef\int_\Omega\rho(t,x) dx \equiv\int_\Omega\rho_0(x) dx,
\end{equation}
which means $\frac{d}{dt}\text{Mass$_\rho$}(t)=0$.

\subsection{The equilibrium state and the macroscopic quantities}
\subsubsection{The equilibrium state and the macroscopic quantities for the VPFP system}\label{subsubsec:VPFP_equilibrium_quantities} It is well-known that the VPFP system has a local equilibrium solution. Bonilla et al. \cite{MR1470927} introduced the form of the steady-state of the VPFP system in bounded domains with the reflection boundary condition on $f(t,x,v)$ and the Dirichlet boundary condtion for the potential $\Phi(t,x)$ without a background charge. They remark that they can prove a result analogously with the Neumann boundary conditions instead of the Dirichlet boundary conditions. In this regard, the VPFP system \eqref{VPFP}, which has the background charge as in \eqref{VPFP}$_3$, has the equilibrium state as follows:
\begin{equation}\label{equilibrium_orginal}
f_{\varepsilon,\infty}(x,v) = C_{vpfp}M(v)e^{-\Phi_\infty(x)},
\end{equation}
where $M(v)\eqdef\frac{1}{\sqrt{2\pi}}e^{-\frac{v^2}{2}}$ is the normalized Maxwellian and $\Phi_\infty(x)$ is a weak solution of the Poisson-Boltzmann (PB) equation:
\begin{equation}\begin{split}\label{PB_VPFP}
-\Delta\Phi_\infty(x)=C_{vpfp}e^{-\Phi_\infty(x)}-h(x),\\
\partial_{x}\Phi_\infty(x)=0\quad\text{on}\quad x\in\partial\Omega,
\end{split}\end{equation}
and
\begin{equation}\notag
C_{vpfp}=\|f_0(\cdot,\cdot)\|_{L^1_{x,v}}\left(\int_\Omega e^{-\Phi_\infty(x)} dx\right)^{-1}.
\end{equation}
The system \eqref{PB_VPFP} is called the Poisson-Boltzmann (PB) equation. By \cite[Theorem 1]{MR1778187}, the system \eqref{PB_VPFP} has a solution $\Phi_\infty(x)$ unique up to an additive constant. If we assume the background charge $h(x)$ as constant, we can check that the system \eqref{PB_VPFP} has a solution $\Phi_\infty(x)=0$ by the global neutrality condition \eqref{VPFP_neutrality}. Therefore, the VPFP system \eqref{VPFP} has the global equilibrium state $(f_{\varepsilon,\infty}, E_{\varepsilon,\infty})$ as
\begin{equation}\label{equilibrium_VPFP}
  f_{\varepsilon,\infty}(x,v)=\frac{\|f_0(\cdot,\cdot)\|_{L^1_{x,v}}}{|\Omega|}M(v), \quad E_{\varepsilon,\infty}(x)=-\partial_x\Phi_\infty(x)=0.
\end{equation}
We expect that the neural network solutions of the VPFP system reach the steady-state \eqref{equilibrium_VPFP} (see simluation Section \ref{subsec:part2_simulation}).

The Lyapunov functional $\eta(t)$ for the VPFP system \eqref{VPFP} is defined by the relative entropy of the solution $f_\varepsilon(t,x,v)$ with respect to a non-normalized stationary distribution $\hat{f}$. As explained in \cite{MR1470927}, we define the Lyapunov functional $\eta(t)$ as
\begin{equation}
	\eta(t)\eqdef\int_\Omega\int_\R f_\varepsilon\log\left(\frac{f_\varepsilon}{\hat{f}_\varepsilon}\right)dvdx,
\end{equation}
where $\hat{f}_\varepsilon$ is defined as
\begin{equation}
	\hat{f}_\varepsilon\eqdef \text{exp}\left\{-\frac{v^2}{2}-\Phi_\varepsilon(t,x)+\frac{1}{\|f_0(\cdot,\cdot)\|_{L^1_{x,v}}}\frac{1}{2}\left(\int_\Omega E_\varepsilon(t,x)^2dx\right)\right\}
\end{equation}
Then, $\eta(t)$ can be reduced as
\begin{align*}
	\eta_\varepsilon(t) &= \int_\Omega\int_\R f_\varepsilon \log\left(\frac{f_\varepsilon }{\hat{f}_\varepsilon }\right)dvdx\\
	&= \int_\Omega\int_\R \left[ f_\varepsilon \log f_\varepsilon  +\frac{1}{2}f_\varepsilon v^2 + f_\varepsilon\Phi - \frac{f_\varepsilon}{2\|f_0(\cdot,\cdot)\|_{L^1_{x,v}}}\left(\int_\Omega E_\varepsilon(t,x)^2dx\right) \right]dvdx\\
	&= \int_\Omega\int_\R f_\varepsilon \log f_\varepsilon  dvdx + \frac{1}{2}\int_\Omega\int_\R f_\varepsilon v^2dvdx + \int_\Omega \left(\int_\R f_\varepsilon dv\right) \Phi_\varepsilon dx\\
	&\quad- \frac{1}{2}\int_\Omega E_\varepsilon (t,x)^2dx
\end{align*}
If we assume the background charge $h(x)$ as constant and assume the zero-mean constraint for the $\Phi_\varepsilon$ (see Remark \ref{remark:zero-mean}), then we have
\begin{multline}\notag
\int_\Omega \left(\int_\R f_\varepsilon dv\right) \Phi_\varepsilon dx
\\= \int_\Omega \left(-\partial_{xx}\Phi_\varepsilon+h(x)\right)\Phi_\varepsilon dx = -\int_\Omega\partial_{xx}\Phi_\varepsilon\Phi_\varepsilon dx + \int_\Omega h(x) \Phi_\varepsilon dx
\\ = -\underbrace{\left[(\partial_{x}\Phi_\varepsilon)\Phi_\varepsilon\right]^{1}_{-1}}_{=0} + \int_\Omega(\partial_{x}\Phi_\varepsilon)^2 dx + \underbrace{\int_\Omega h(x) \Phi_\varepsilon dx}_{=0}= \int_\Omega E_\varepsilon (t,x)^2dx.
\end{multline}
Therefore, it yields that
\begin{multline}\notag
	\eta_\varepsilon(t)= \int_\Omega\int_\R f_\varepsilon \log f_\varepsilon  dvdx + \frac{1}{2}\int_\Omega\int_\R f_\varepsilon v^2dvdx + \frac{1}{2}\int_\Omega E_\varepsilon (t,x)^2dx
	\\\eqdef -\text{Ent}(t) + \text{KE}(t) + \text{EE}(t),
\end{multline}
where the entropy of the system ``\text{Ent}", the total kinetic energy ``\text{KE}", and the electric pontential energy ``\text{EE}" of the system are defined as
\begin{equation}\label{Ent}
	\text{Ent}(t)\eqdef- \int_{\Omega\times \R} f_\varepsilon\log f_\varepsilon dxdv,
\end{equation}
\begin{equation}\label{KE}
	\text{KE}(t) \eqdef \frac{1}{2} \int_{\Omega\times \R} |v|^2f_\varepsilon dxdv,
\end{equation}
and
\begin{equation}\label{EE}
	\text{EE}(t)\eqdef \frac{1}{2}\int_\Omega |E_\varepsilon|^2 dx.
\end{equation}
The Lyapunov functional is also called the free energy defined as
\begin{equation}\label{FE}
	\text{FE}(t)\eqdef -\text{Ent}(t) + \text{KE}(t) + \text{EE}(t).
\end{equation}
Since the Lyapunov functional satisfies $\frac{d}{dt}\eta(t)\leq0$, we expect that the free energy \eqref{FE} is a non-increasing function (see Section \ref{subsec:part2_simulation}).

\subsubsection{The equilibrium state and the free energy for the PNP system}\label{subsubsec:PNP_equilibrium_quantities}The steady state of the PNP system \eqref{PNP} satisfies
$$\partial_x(\partial_x\rho_\infty-\rho_\infty E_\infty)=0,$$
from the equations \eqref{PNP}$_1$. It is reduced to
$$\frac{\partial_x\rho_\infty}{\rho_\infty}-E_\infty=\text{Constant}.$$
Therefore, we have the following steady state:
\begin{equation}\label{steady_state}
\rho_\infty(x)=C_{pnp}\exp\left(\int_\Omega E_\infty(x)dx\right),
\end{equation}
where $E_\infty(x)$ is a solution of the Poisson-Boltzmann (PB) equation
\begin{equation}\begin{split}\label{PB_PNP}
\partial_xE_\infty(x)&=C_{pnp}\exp\left(\int_\Omega E_\infty(x)dx\right)-h(x),\\
E_\infty(x)&=0\quad\text{on}\quad x\in\partial\Omega,
\end{split}\end{equation}
with some constant $C_{pnp}$. We can express the constant $C_{pnp}$ using the total density \eqref{mass rho conservation} as
\begin{equation}\label{PNP_constant}
C_{pnp}=\left(\int_{[-1,1]}\rho_0(x)dx\right)\left(\int_{[-1,1]} \exp\left(\int_\Omega E_\infty(x) dx\right) dx\right)^{-1}.
\end{equation}
The PB equation \eqref{PB_PNP} has a solution $E_\infty(x)=0$ similar to the PB equation \eqref{PB_VPFP} in the VPFP system. Therefore, the PNP system \eqref{PNP} has the steady state as follows:
\begin{equation}\label{equilibrium_PNP}
  \rho_\infty(x)=C_{pnp}, \quad E_\infty(x)=0,
\end{equation}
with the constant $C_{pnp}$ which is defined in \eqref{PNP_constant}.
We expect that the neural network solutions of the PNP system reach the steady-state \eqref{equilibrium_PNP} (see simluation Section \ref{subsec:part3_simulation}).

Also, the free energy $\text{FE$_\rho$}(t)$ of the PNP system \eqref{PNP} is defined as follows (similar to \cite{flavell2014conservative,MR2815679,MR3571893}):
\begin{equation}\label{PNP_FE}
  \text{FE$_\rho$}(t)\eqdef\int_\Omega\left(\rho(t,x)\log\rho(t,x) +\frac{1}{2}E(t,x)^2 \right)dx,
\end{equation}
which has both the entropic part and the interaction part. The first term\newline $\rho(t,x)\log\rho(t,x)$ on the right-hand side is the entropy related to the Brownian motion of each particles, and the second term $\frac{1}{2}E(t,x)^2$ is the electric potential energy of the particles.

Under the specific boundary conditions \eqref{no-flux} and \eqref{PNP}$_5$, the PNP system has the following relation:
\begin{multline}\notag
  \int_\Omega\Phi(\rho(t,x)-h(x)) dx = \int_\Omega\Phi\partial_x E dx= -\int_\Omega\Phi\partial_{xx}\Phi dx\\
  =-[\Phi\partial_x\Phi]^1_{-1}+\int_\Omega(\partial_{x}\Phi)^2 dx = \int_\Omega(\partial_{x}\Phi)^2 dx = \int_\Omega E(t,x)^2 dx,
\end{multline}
by multiplying $\Phi(t,x)$ onto \eqref{PNP}$_3$, integrating it over the domain $\Omega$ and using the integration by parts with respect to $x$. Therefore, the free energy can be rewritten as
\begin{equation}\notag
  \text{FE$_\rho$}(t)=\int_\Omega\left(\rho(t,x)\log\rho(t,x) +\frac{1}{2}(\rho(t,x)-h(x))\Phi \right)dx.
\end{equation}
By taking the time derivative of the free energy $F(t)$, we can derive
\begin{multline}\notag
  \frac{d}{dt}\text{FE$_\rho$}(t) = \int_\Omega\left(\rho_t\log\rho+\rho_t\right) dx +\frac{1}{2}\int_\Omega \left((\rho-h(x))\Phi_t + \rho_t\Phi\right) dx\\
  = \int_\Omega\rho_t\left(\log\rho+1+\Phi\right) dx +\frac{1}{2}\int_\Omega \left((\rho-h(x))\Phi_t - \rho_t\Phi\right) dx.
\end{multline}
Then, we have 
\begin{multline}\notag
  \frac{d}{dt}\text{FE$_\rho$}(t) = \int_\Omega\partial_x(\rho_x-\rho E)\left(\log\rho+1+\Phi\right) dx+\frac{1}{2}\int_\Omega \left(-\Phi_{xx}\Phi_t + \Phi_{txx}\Phi\right) dx\\
  =-\int_\Omega(\rho_x-\rho E)\left(\frac{\rho_x}{\rho}+\Phi_x\right) dx + \frac{1}{2}\int_\Omega\partial_x\left(\Phi\Phi_{tx}-\Phi_t\Phi_x \right)dx\\
  =-\int_\Omega\frac{1}{\rho}(\rho_x-\rho E)^2 dx + \frac{1}{2}\left[\Phi\Phi_{tx}-\Phi_t\Phi_x \right]^1_{-1}=-\int_\Omega\frac{1}{\rho}(\rho_x-\rho E)^2 dx.
\end{multline}
Therefore, the PNP system \eqref{PNP} satisfies the following free energy dissipation law:
\begin{equation}\label{FE_dissipation}
  \frac{d}{dt}\text{FE$_\rho$}(t)=-\int_\Omega\frac{1}{\rho}(\rho_x-\rho E)^2 dx\leq0.
\end{equation}
We expect that the free energy \eqref{PNP_FE} of the PNP system is a non-increasing function (see simulation Section \ref{subsec:part3_simulation}).

\subsection{Mathematical results on the diffusion limit}In this section, we introduce past results on the diffusion limit of the VPFP system. There are two scalings of the VPFP. The first one is the diffusion limit (or the parabolic limit, or the low field limit), and the second one is the drift limit (or the hyperbolic limit, or the high field limit). In this paper, we consider the first one on the diffusion limit of the VPFP system only. The diffusion limit has been extensively investigated in many works. Poupaud in \cite{MR1127004} considers the diffusion limit of the semiconductor Boltzmann equation. The diffusion limit for the VPFP system with a given background was considered by \cite{MR1780148,MR2139941} in the two-dimensional case. And later, El Ghani and Masmoudi in \cite{MR2664460} extend these results to higher dimensional cases in the renormalized sense. The case of multiple-species dynamics is also considered in \cite{MR3294407,MR3479195}. A recent paper \cite{MR3294407} of Wu, Lin and Liu treats the diffusion limit of the VPFP system in a bounded domain with reflection boundary conditions. We used the results of this paper to show the bottom side of Figure \ref{fig:diagram}. Also, there are many works that deal with the drift limit as in \cite{MR2142374,MR1859832,MR1848592,MR1834113}.

\subsection{Existing numerical methods and an Asymptotic Preserving scheme}
In this section, we introduce a brief history of the numerical methods to approximate the solutions of the VPFP system and the PNP equation. We also introduce the numerical studies concerning the asymptotic expansions on the small parameters, the so-called Asymptotic Preserving (AP) scheme.

There are many numerical studies to solve the VPFP system and related systems. There is a wide range of literature on numerical analysis for the Fokker-Planck (FP) equation including the finite difference method \cite{MR1739113,MR1619879}, and its conservative type scheme \cite{MR892257,MR1283340,MR1640174,MR1688993,MR1447091}. The particle method \cite{allen1994computational,MR1377255} is an effective method for the stochastic properties of the Fokker-Planck operator. Also, Wollman and Ozizmir \cite{MR2145394,MR2392562,MR2567867} provided the deterministic particle method for the VPFP systems in one and two-dimensional cases. Another approach is the spectral method to solve the Fokker-Planck equation. In \cite{MR1795398}, they develop a new spectral method based on a Fourier spectral approximation for the Boltzmann equation. Filbet and Pareschi in \cite{MR1906573} extended the method to the nonhomogenous case. The review paper \cite{MR3202241} contains the latest references on numerical methods for collisional kinetic equations.

Also, a lot of efforts have been made to the numerical methods for the PNP system. Many of the existing methods have been constructed for both one-dimensional and higher dimensional cases in various chemical and biological contexts. We refer to some recent studies for solving time-dependent PNP systems. Solkalski et al. in \cite{sokalski2003numerical} proposed the finite difference scheme for analyzing liquid junction and ion-selective membrane potentials. Hyon et al. in \cite{MR2815679} provided another finite element method with the back-Euler method for the modified PNP system. It is considered to be difficult for numerical schemes to provide the physical properties of the PNP system; namely the nonnegativity principle, the mass conservation, and the free energy dissipation. Regarding these difficulties, Liu and Wang in \cite{MR3192448} developed a finite difference method for the PNP system. They focus on the development of a free energy satisfying numerical method for the PNP system. They also provided the discontinuous Galerkin scheme for a one-dimensional case in \cite{MR3571893}. The implicit methods with the trapezoidal rule and backward differentiation are presented in \cite{flavell2014conservative}.

Regarding the numerical methods to capture the relation between two regimes, Shi Jin in \cite{MR1718639} first introduced the numerical scheme that preserves the asymptotic limits from the mesoscopic to the macroscopic models for transport in diffusive regimes - the asymptotic-preserving (AP) scheme. This scheme could be illustrated in the following commutative diagram of \cite[Figure 1]{MR2964096}:
\begin{figure}[H]
  \includegraphics[width=0.3\textwidth, draft=false]{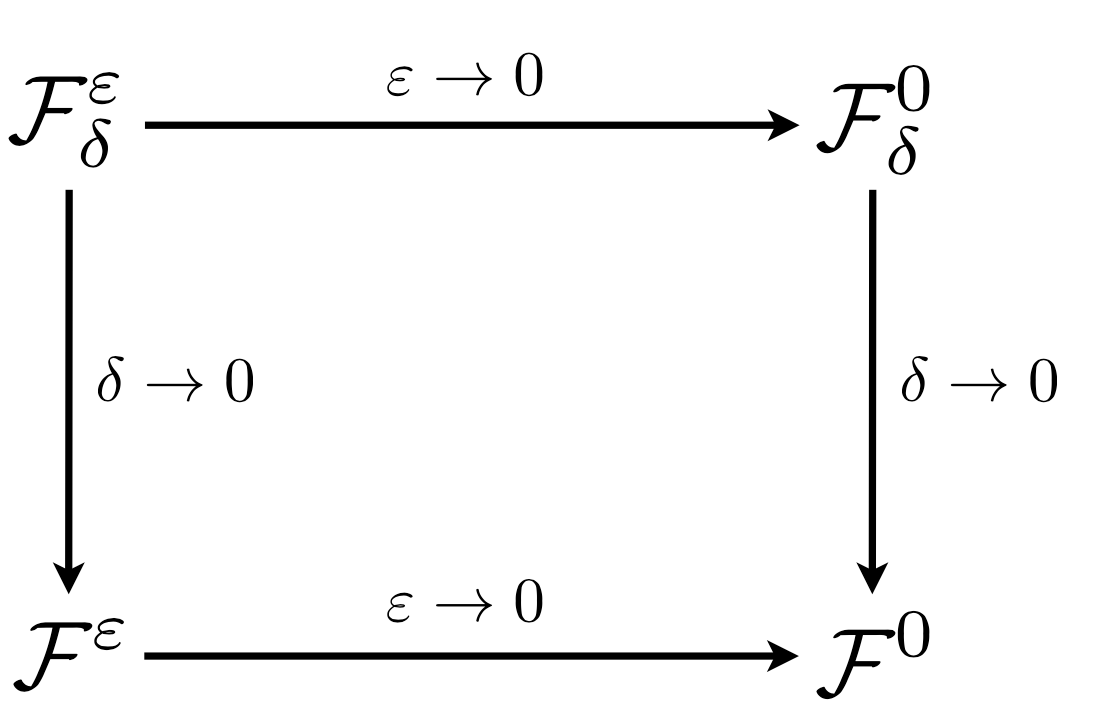}
  \caption{Illustration of AP schemes}
  \label{fig:AP}
\end{figure}
As explained in \cite{MR2964096}, $\mathcal{F}^\varepsilon$ is a mesoscopic model which depends on parameter $\varepsilon$ that characterizes the small scale. $\mathcal{F}^\varepsilon_\delta$ is a discretization of $\mathcal{F}^\varepsilon$ with parameter $\delta$ that is related to numerical discretization (such as mesh size and/or time step). As $\varepsilon$ goes to zero, the mesoscopic model $\mathcal{F}^\varepsilon$ is approximated by a macroscopic model $\mathcal{F}^0$. Then, the scheme $\mathcal{F}^\varepsilon_\delta$ is called AP if the asymptotic limit of $\mathcal{F}^\varepsilon_\delta$ as $\varepsilon\rightarrow0$ with $\delta$ fixed, denoted by $\mathcal{F}^0_\delta$, is a good approxaimation of $\mathcal{F}^0$.

The AP schemes are developed for various equations. Especially, there are many studies that deal with the AP schemes for the kinetic equations with the Euler regime. Filbet and Jin in \cite{MR2674294} developed a penalization method to overcome the Boltzmann integral, which is a fully nonlinear collision operator. Jin and Yan generalized their idea to the nonhomogeneous Fokker-Planck-Landau equation. Dimarco and Pareschi in \cite{MR2861709} introduce an exponential Runge-Kutta method for kinetic equations. The AP schemes for the high field limit of the VPFP system are considered in \cite{MR2931501,MR2786393}. In \cite{MR2786393}, they also developed the AP scheme based on a micro-macro decomposition for the diffusion limit of the Vlasov-Poisson-BGK model. We refer to the recent surveys by Jin \cite{MR2964096}, Degond \cite{MR3220424} and Pareschi and Russo \cite{pareschi2011efficient}.

Given the existing numerical methods in the literature, the main distinctive of this paper is the use of the neural network approach as a solver for these important problems. We used the neural network method as a function approximator for the VPFP system, PNP system, and the AP scheme as the parameter $\epsilon$ goes to 0. The aim of this paper is to complete Figure \ref{fig:diagram} of neural network vesion similar to Figure \ref{fig:AP} of the numerical analysis version. 

\subsection{Neural network and an approach to solve a PDE}
Neural network has also drawn attention in the machine learning community. It has been used for various fields such as natural language process, image recognition, speech recognition, and others. Deep Learning, which uses a deep stack of neural network layer called a Deep Neural Network (DNN), is effectively applied in these areas. The neural network architecture is introduced in \cite{MR10388} for the first time. There are theoretical results to justify the use of neural networks in these applications such as \cite{MR1015670,hornik1989multilayer,funahashi1989approximate,hornik1991approximation}. The key theorem to these results is the universal approximation theorem. The universal approximation theorem states that an arbitrary real-valued function can be well-approximated by a feed-forward neural network. Later, Li in \cite{li1996simultaneous} showed that the neural network with one hidden layer could approximate not only a target function but also its higher partial derivatives on a compact set.

 Then Deep Learning as a PDE solver has also been studied; \cite{lagaris1998artificial,lagaris2000neural} suggested the use of neural networks to solve ODEs and PDEs. Recently, Raissi et al. \cite{MR3881695} introduced physics informed neural networks. They design data-driven algorithms for two main problems: data-driven solutions and data-driven discovery of partial differential equations. The data-driven method to solve the high-dimensional PDEs with a DNN is proposed in \cite{MR3874585}. The second problem, called the forward-inverse problem, is also considered in \cite{MR4112187} with a theoretical analysis of the convergence of the DNN solutions to the classical solutions. In \cite{anitescu2019artificial}, they present a method for approximating the solution of PDEs using an adaptive collocation strategy. Also, Han et al. in \cite{MR4030583} deal with the uniformly accurate moment system using the kinetic equation as an example.

Hwang et al. in \cite{MR4116803} introduce the Deep Neural Network solutions to the kinetic Fokker-Planck equation in a bounded interval under the varied types of the physical boundary conditions. They observed the asymptotic behaviors of the DNN solutions to verify an agreement with theoretical results. They also provide the theoretical proofs on the relationship between the DNN solutions and the a priori analytic solutions. Our paper is motivated by several ideas in \cite{MR4116803}. We expand their ideas to a more general VPFP system and its diffusion limit.

\subsection{Outline of the paper}Each of the four sides of Figure \ref{fig:diagram} consists of four parts (Part I, II, III, and IV). In Section \ref{sec:part1} (Part I), we show that the solutions of the VPFP system converge to the solutions of the PNP system as the Knudsen number $\varepsilon$ tends to zero, which corresponds to the bottom side of Figure \ref{fig:diagram}. To this end, we use the theoretical result from the paper \cite{MR3294407}. In Section \ref{sec:methodology}, we will introduce in detail our Deep Learning method to approximate the solution of the VPFP system and the solution of the PNP system, which is used for the numerical simulations in Part II, III, and IV. Part II will include the detailed descriptions on the DNN architectures for each system (Section \ref{subsec:dnn_architerture}), the definition of grid points (Section \ref{subsec:grid}), and a `Grid Reuse' method that is a newly devised tool in the paper (Section \ref{subsec:grid_reuse}) to capture the dynamics under a small Knudsen number $\varepsilon$.
In Section \ref{sec:part2} (Part II), we will introduce the DNN approximated solutions to the VPFP system \eqref{VPFP}, which corresponds to the left-hand side of Figure \ref{fig:diagram}. We will provide the suitable loss functions \eqref{VPFP_loss_total} to approximate the VPFP system using the Deep Learning in Section \ref{subsec:part2_loss}. We will prove the convergence of the DNN solution to an analytic solution of the VPFP system as the loss function vanishes in Section \ref{subsec:part2_theory}. We will also provide the numerical simulations that show the asymptotic behaviors of macroscopic quantities and the pointwise values of the DNN solution to the VPFP system in Section \ref{subsec:part2_simulation}.
In Section \ref{sec:part3} (Part III), we will introduce the DNN approximated solutions to the PNP system \eqref{PNP}, which corresponds to the right-hand side of Figure \ref{fig:diagram}. The contents would be analogous to those in Section \ref{sec:part2}. In Section \ref{sec:part4} (Part IV), we will provide several numerical simulations to see the trend of the diffusion limit from the VPFP system to the PNP system, which corresponds to the upper side of Figure \ref{fig:diagram}. We will analyze the convergence \eqref{converge_f} and \eqref{converge_E} using the DNN solutions of the VPFP system by varying the Knudsen number from 1 to 0.05 via the Asymptotic-Perserving (AP) scheme. Finally, in Section \ref{sec:conclusion}, we will summarize our methods and the results.

\section{Part I. On convergence of the VPFP solution to the PNP solution}\label{sec:part1}
In this section, we introduce the convergence of solutions of the VPFP system to a solution of the PNP system from the recent paper \cite[Theorem 2.1]{MR3294407}. Wu, Lin and Liu \cite{MR3294407} prove that the VPFP system \eqref{VPFP} with the Maxwellian reflection boundary condition converges to the PNP system \eqref{PNP} as $\varepsilon$ tends to zero for the multi-species model case. To be more specific, they consider the renormalized solution $(f_{\varepsilon,i},\Phi_\varepsilon)$ of a rescaled $N$-species VPFP system ($i=1,2,...,N$) in a bounded interval $\Omega\subset\R^d$ using the scaled parameters as
\begin{equation}\label{Wu_VPFP}\begin{split}
\partial_t f_{\varepsilon,i}+\frac{1}{\varepsilon} v\cdot\nabla_xf_{\varepsilon,i}-\frac{\kappa_iz_i}{\varepsilon} \nabla_x\Phi_\varepsilon\cdot\nabla_vf_{\varepsilon,i} &= \frac{\zeta_i}{\varepsilon^2}\nabla_v\cdot(vf_{\varepsilon,i}+ \kappa_i\nabla_v f_{\varepsilon,i}),\\
-\varpi\Delta_x\Phi_\varepsilon &= \sum_{i=1}^Nz_i\int_{\R^d}f_{\varepsilon,i}(t,x,v)dv-h(x),
\end{split}
\end{equation}
with initial condition, reflection boundary condition (especially, Maxwellian boundary condition) for the distribution function $f_{\varepsilon,i}$ and zero-outward electric field condition (Neumann boundary condition) for the electric potential $\Phi_\varepsilon$. They show that the solution $(f_{\varepsilon,i},\Phi_\varepsilon)$ converges to $(\rho_i(t,x)M_i(v),\Phi(t,x))$, where $(\rho_i, \Phi)$ is a weak solution of the PNP system in a bounded interval $\Omega$ as
\begin{equation}\label{Wu_PNP}\begin{split}
\partial_t \rho_i + \nabla_{x}\cdot\overbrace{\left( -\frac{1}{\zeta_i}\nabla_x\rho_i-\frac{z_i}{\zeta_i}\rho_i\nabla_x\Phi \right)}^{J_i(t,x)\eqdef}&=0,\\
-\varpi\Delta_x\Phi&=\sum_{i=1}^Nz_i\rho_i-h(x),
\end{split}
\end{equation}
with the initial-boundary conditions given as follows
\begin{equation}\label{Wu_bdry}\begin{split}
J_i\cdot n=0,\ \text{on} \ \partial\Omega,\\
\nabla_x \Phi\cdot n =0,\ \text{on} \ \partial\Omega,
\end{split}
\end{equation}
as $\varepsilon$ tends to zero ($M_i=M_i(v)$'s are the normalized Maxwellians for each species).

Using this result, we derive our specific system \eqref{VPFP} with the boundary condition. Firstly, we specify the 1-dimension bounded domain $\Omega=(-1,1)\subset\R$ on the spatial domain and $\R$ on the velocity domain. Also, we consider the single-species case with $N=1$. This case is reasonable in plasma physics when the relatively huge ions are supposed to be static in the background. In this case, we denote the distribution function $f_{\varepsilon,i}$ as $f_\varepsilon$ for the VPFP system, since $i=1$. 
We also choose the classical specular reflection boundary condition for the $f_\varepsilon(t,x,v)$ instead of the Maxwellian boundary condition used in \cite{MR3294407}. We use the Dirichlet boundary condition for the electric force $E_\varepsilon(t,x)$ which is the same as the Neumann boundary condition for the electric field $\Phi_\varepsilon(t,x)$ assumed in \cite{MR3294407}. The boundary conditions \eqref{Wu_bdry} imply the Neumann condition ($\nabla_x\rho_i\cdot n=0$ on $\partial\Omega$) for the density function and the Dirichlet condition ($E=\nabla_x\Phi\cdot n=0$ on $\partial\Omega$) for the electric force $E(t,x)$. Additionally, we set all the parameters to be 1 in the systems \eqref{Wu_VPFP} and \eqref{Wu_PNP} except the Knudsen number $\varepsilon$ to take the limit. 

Then, the solution $f_\varepsilon$ (corresponding to the $f_{\varepsilon,i=1}$) and the solution $E_\varepsilon=\partial_x\Phi_\varepsilon$ to the VPFP system \eqref{VPFP} with the specular boundary condition \eqref{specular} satisfy the following convergence:
\begin{equation}\label{converge_f}
	f_\varepsilon(t,x,v) \rightarrow \rho(t,x) M(v)\quad\text{in}\quad L^1(0,T;L^1(\Omega\times\R)),
\end{equation}
where $M(v)=\frac{1}{\sqrt{2\pi}}e^{-\frac{v^2}{2}}$ and
\begin{equation}\label{converge_E}
	E_\varepsilon(t,x) \rightarrow E(t,x)\quad\text{in}\quad L^2(0,T;L^p(\Omega)),\quad1\leq p<2
\end{equation}
as the Knudsen number $\varepsilon$ tends to zero, where the density $\rho$ (corresponding to the $\rho_{i=1}$) and the solution $E$ satisfy the system \eqref{PNP} with the no-flux boundary condition \eqref{no-flux}. In the Part IV (Section \ref{sec:part4}) of this paper, we provide the corresponding numerical simulations which show the trend of the convergence \eqref{converge_f} and \eqref{converge_E}. 

\begin{remark}\label{remark:zero-mean} In \cite{MR3294407}, they prove the diffusion limit with two assumption for $\Phi_\varepsilon$ on the Poisson equation; the global neutrality condition and the zero-mean constraint. The global neutrality condition is the same as the condition \eqref{VPFP_neutrality} we assumed. They also assume the zero-mean constraint as follows:
\begin{equation}\label{zero-mean}
	\int_\Omega\Phi_\varepsilon dx=0.
\end{equation}
This constraint is necessary to uniquely determine the solution $\Phi_\varepsilon$. However, we are interested in the solution $E_\varepsilon(t,x)=-\partial_x\Phi_\varepsilon(t,x)$ which is the partial of $\Phi(t,x)$ instead of $\Phi_\varepsilon(t,x)$ in this paper. Without loss of generality, we can assume \eqref{zero-mean} to apply the diffusion limit theorem from \cite{MR3294407}.
\end{remark}

\section{Simulation methodology: The Deep Learning approach}\label{sec:methodology}
In this section, we introduce our deep learning method to solve the Cauchy problem to the Vlasov-Poisson-Fokker-Planck system \eqref{VPFP} and the Poisson-Nernst-Planck (PNP) system \eqref{PNP}.

\subsection{A Deep Learning approach for solving partial differential equation}\label{subsec:method}
A Deep Learing algorithm can be described in terms of a non-linear function approximation method using a 
Deep Neural Network (DNN). A Deep Neural Network consists of a sequence of multiple layers. Each layers has several neurons, which receive the neuron activation from the pre-layer as input. The neurons implement the weighted sum of the input and apply an activation function in order to transform the output to a non-linear one. The output is transmitted to neurons in the post-layer. We assume that a DNN has $L$ layers; it has an input layer, $L-1$ hidden layers and an output layer ($L-$th layer). Similarly to the explanation of \cite{MR4116803}, we denote the relation between the $l-$th layer and the $(l+1)-$th layer ($l=1,2,...,L-1$) as
\begin{equation*}
	z_j^{(l+1)}=\sum_{i=1}^{m_l}w_{ji}^{(l+1)}\bar{\sigma}_{l}(z_i^l) + b_j^{(l+1)},
\end{equation*}
where $m=(m_0,m_1,m_2,...,m_{L-1})$, $w=\{w^{(k)}_{ji}\}_{i,j,k=1}^{m_{k-1},m_k,L},$ $b=\{b^{(k)}_j\}_{j=1,k=1}^{m_k,L},$ and
\begin{itemize}
	\item $z_i^l$ : the $i$-th neuron in the $l$-th layer
	\item $\bar{\sigma}_l$ : the activation function in the $l$-th layer
	\item $w_{ji}^{(l+1)}$ : the weight between the $i$-th neuron in the $l$-th layer and the $j$-th neuron in the $(l+1)$-th layer
	\item $b_j^{(l+1)}$ : the bias of the $j$-th neuron in the $(l+1)$-th layer
	\item $m_l$ : the number of neurons in the $l$-th layer.
\end{itemize}
Note that the relation between the input layer and the first hidden layer is expressed as follows:
\begin{equation*}
	z_j^{1}=\sum_{i=1}^3 w_{ji}^1 z_i^0 + b_j^{1},
\end{equation*}
where $(z_1^0, z_2^0, z_3^0)=(t,x,v)$.

The deep learning algorithm learns the complex nonlinear mapping by adapting these weights $w_{ji}^{(l+1)}$ and biases $b_j^{(l+1)}$ to make the output of Deep Neural Network similar to the target function, in our case, the solution of the VPFP and PNP system. The Deep learning uses the back-propagation learning algorithm, which applies the chain rule to calculate the influence of each weight and each bias to reduce a pre-defined cost function, which is called ``loss function" in the Deep Learning. Then, the algorithm uses the gradient method to update the weights and biases.

To approximate a solution of PDEs using the deep learning algorithm, we need an appropriate loss function with respect to the PDE system. For example, suppose we coinsider the following parabolic PDE:
\begin{align*}\notag
\frac{\partial u}{\partial t}(t,x)+\mathcal{L}u(t,x)&=0,\hspace{3mm}(t,x)\hspace{1mm}\text{in}\hspace{1mm}[0,T]\times\Omega,
\\u(0,x)&=u_0(x),\hspace{3mm}x\hspace{1mm}\text{in}\hspace{1mm}\Omega,
\\\mathcal{B}u(t,x)&=q(x),\hspace{3mm}(t,x)\hspace{1mm}\text{in}\hspace{1mm}[0,T]\times\partial\Omega,
\end{align*}
where $\mathcal{L}$ is a differential operator and $\mathcal{B}$ is the boundary operator with known functions $u_0(x)$ and $q(x)$. In many  papers (e.g. \cite{MR3874585,MR3881695}), they approximate the solution $u(t,x)$ using the DNN output $u^{nn}(t,x)$ with the loss function as
\begin{multline}
	Loss(u^{nn})=\left\|\frac{\partial u^{nn}}{\partial t}(t,x)+\mathcal{L}u^{nn}(t,x)\right\|^2_{L^2([0,T]\times\Omega)}\\+\left\|u^{nn}(0,x)-u_0(x)\right\|^2_{L^2(\Omega)}+\left\|\mathcal{B}u^{nn}(t,x)-q(t,x)\right\|^2_{L^2([0,T]\times\partial\Omega)}.
\end{multline}

The proposed loss function is an intuitive one to approximate the solution of PDE. In our case, we propose slightly different loss functions for each system. We define the loss function for the VPFP system in Section \ref{subsec:part2_loss} (Part II) and for the PNP system in Section \ref{subsec:part3_loss} (Part III). We propose the loss functions based on our theoretical evidence. In each section, we prove that the DNN output to the VPFP system and PNP system converges to a \textit{priori} classical solution to each system if the proposed loss function goes to zero. The details are precisely described in Part II and Part III.

\subsection{Our Deep Learning algorithm and the architecture}\label{subsec:dnn_architerture}
We take two different neural network structures which share the same inputs to approximate the coupleded nonlinear equations. Each DNN has four hidden layers and each layer has 3(or 2)-100-100-100-100-1 neurons. For the VPFP system, the two Deep Neural Networks are used to approximate the solutions, $f$ and $E$, respectively. The neural network structure is precisely shown in Figure \ref{fig:DNN_structure_VPFP}.
\begin{figure}[H]
  \includegraphics[width=0.8\textwidth, draft=false]{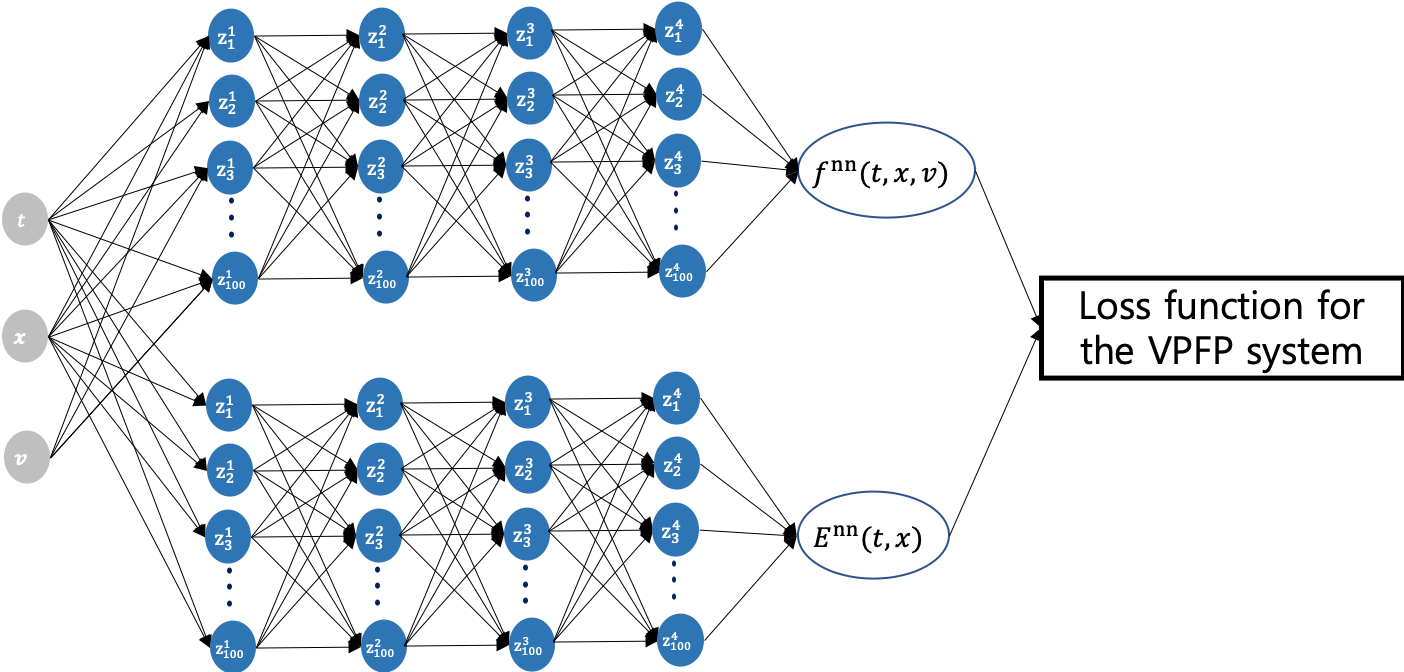}
  \caption{The DNN structure for the VPFP system}
  \label{fig:DNN_structure_VPFP}
\end{figure}
We denote the approximated solution as $(f_\varepsilon^{nn}(t,x,v;m,w,b),E_\varepsilon^{nn}(t,x;m,w,b))$, which consists of the output of each DNN. The two outputs $f_\varepsilon^{nn}(t,x,v;m,w,b)$ and $E_\varepsilon^{nn}(t,x;m,w,b)$ are used to calculate the pre-defined loss function. Then, we use a gradient descent algorithm to update the weights and biases of our model's parameters by iteratively moving in the direction of reducing the loss function. In this work, we use the Adam (Adaptive Moment Estimation) optimizer, an efficient variant of the stochastic gradient descent algorithm which is widely used in deep learning applications due to quick convergence in training. 

Similarly, we use the two Deep Neural Networks to approximate the solutions $(\rho,E)$ for the PNP system as in Figure \ref{fig:DNN_structure_PNP}. We denote the approximated solution as $(\rho^{nn}(t,x;m,w,b),E^{nn}(t,x;m,w,b))$.
\begin{figure}[H] 
  \includegraphics[width=0.8\textwidth, draft=false]{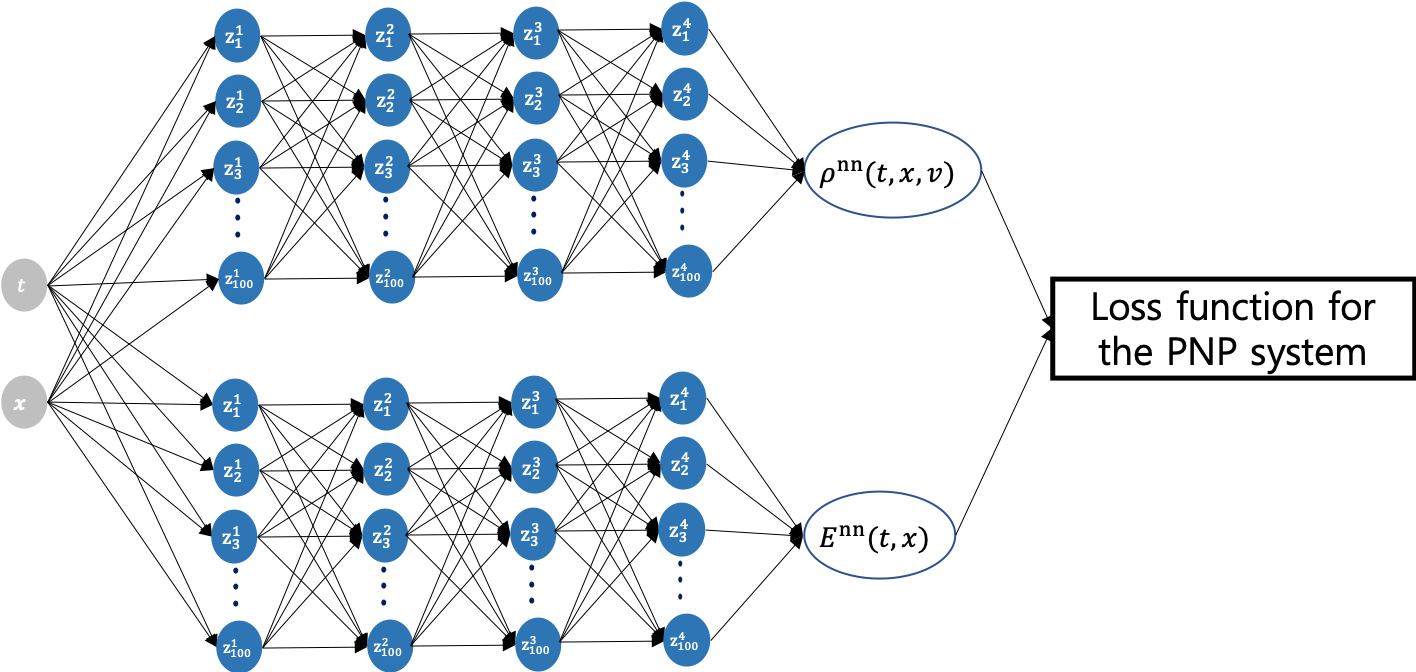}
  \caption{The DNN structure for the PNP system}
  \label{fig:DNN_structure_PNP}
\end{figure}

For the four hidden layers in each DNN, we use the hyper-tangent activation function ($\bar{\sigma}(x)=\frac{e^x-e^{-x}}{e^x+e^{-x}}$), which is the common activation function in Deep Learning literature. While the choice of the activation function for the hidden layers is quite clear, the choice of an activation function for the output layer depends on the purposes. We use the Softplus activation function ($\bar{\sigma}(x)=\ln (1+e^x)$) only for the output $f_\varepsilon^{nn}(t,x,v;m,w,b)$. It is one of the main issues to preserve the positivity of the output when the numerical scheme is constructed. Since the Softplus function has outputs in scale of $(0,+\infty)$, we easily apply the positivity constraint for the output $f_\varepsilon^{nn}(t,x,v;m,w,b)$.

We use the PyTorch library for deep learning. It is one of the most standard deep learning frameworks due to its simplicity and ease of use. We also use the Adam optimizer in PyTorch library with the Learning rate scheduling, which adjusts the learning rate based on the number of epochs. Regarding the loss function, we need the derivation and integration of the output with respect to the variables $t$, $x$ and $v$. To approximate the derivatives of the neural network output with respect to the input variables, we use the Autograd package in PyTorch library. It provides Automatic Differentiation (AD), which is one of the powerful techniques in scientific computing. The AD is different from the usual differentiation methods, such as numerical differentiation or the symbolic differentiation. We refer to the survey papers \cite{paszke2017automatic,MR3800512} for more details. Also, we use the trapezoidal rule from the PyTorch library to approximate the integration. The specific loss functions we defined for the VPFP system and the PNP system are explained in Part II (VPFP) and Part III (PNP).

\subsection{Training data: grid points}\label{subsec:grid}
To approximate the solutions to the VPFP system and the PNP system via the Deep Learning algorithm, we make the grid points for each variable domain as inputs in the neural networks. We need three-dimensional time-space-velocity grid for the probability density $f_\varepsilon^{nn}(t,x,v;m,w,b)$ in VPFP system and two-dimensional time-space grid for the density $\rho^{nn}(t,x)$ in PNP system and the force field $E_\varepsilon^{nn}(t,x;m,w,b)$ to the VPFP system and $E^{nn}(t,x;m,w,b)$ to the PNP system. We choose the time interval $[0,T]$ as $[0,5]$ only for the VPFP system with $\varepsilon=1$ and $[0,1]$ with the smaller Knudsen number $\varepsilon$, which is enough to see the steady-state of both the VPFP system and the PNP system. Also, we truncate the momentum space for the $v$ variable as $V\eqdef[-10,10]$ for training and assume that $f_\varepsilon^{nn}(t,x,v;m,w,b)$ is $0$ if $|v|>10$. Note that we sample the grid points for the each variables $t$, $x$ and $v$ randomly for each iteration. Compared to the grid created by dividing the domain uniformly, this sampling-based approach has the effect of selecting infinite grids in the each domain. More precisely, the grid points for training $f_\varepsilon^{nn}(t,x,v;m,w,b)$ are chosen randomly as follows:
\begin{equation}\label{data_ge}
	\left\{ (t_i,x_j,v_k) \right\}_{i,j,k} \in [0,T] \times \Omega \times V
\end{equation}
for the governing equation,
\begin{equation}\label{data_ini}
	\left\{ (t=0,x_j,v_k) \right\}_{j,k} \in \Omega \times V
\end{equation}
for the initial condition and
\begin{equation}\label{data_bdry}
	\left\{ (t_i,x=-1\text{ or }1,v_k) \right\}_{i,k} \in [0,T] \times  V
\end{equation}
for the boundary condition with $T=1$ or $T=5$, $\Omega=[-1,1]$ and $V=[-10,10]$. In every epoch, we sample the grid points for the time $t$ as $\{t_i\}_{i=1}^{10}$, for the position $x$ as $\{x_j\}_{j=1}^{10}$, and for the velocity $v$ as $\{v_k\}_{k=1}^{1000}$. We use a larger number of velocity grids than the time and position grids to approximate the integration with respect to the velocity in the VPFP system \eqref{VPFP}$_3$. We can choose the grid points similarly for $\rho^{nn}(t,x)$, $E_\varepsilon^{nn}(t,x)$, and $E^{nn}(t,x)$. 

\subsection{`Grid Reuse' method to capture the small Knudsen number}\label{subsec:grid_reuse}
In Part IV, we provide the numerical simulations when the Knudsen number $\varepsilon$ is small. It is hard to capture the asymptotic limit with the fixed numerical discretization in numerical schemes. 

To overcome this challenge, we propose a newly devised technique in this paper; we call it `Grid Reuse' method. The Deep Neural Network is trained to minimize the sum of loss functions at randomly sampled grid points in every epoch, as explained in \eqref{data_ge}, \eqref{data_ini}, and \eqref{data_bdry}. The idea of our `Grid Reuse' method is that we add more top$-k$ grid points of these randomly sampled grid points to use for training in the next epoch. The top$-k$ grid points mean that the grid points have the largest top$-k$ values of every epoch's loss function.  The `Grid Reuse' method helps to solve the time dependency, which is the main difficulty of capturing the diffusion limit. In order to use the trapezoidal rule to approximate $\int_V f_\varepsilon^{nn}(t,x,v;m,w,b) dv$, we only catch the time and position grid points where loss function has the largest values. We then make the three-dimensional time-space-velocity grid with a randomly sampled velocity grid in $V=[-10,10]$.

The `Grid Reuse' method is inspired by the Residual-based adaptive refinement (RAR) method in \cite{lu2019deepxde} and the adaptive collocation method in \cite{anitescu2019artificial}. The technique of these methods and our method are similar to the adaptive mesh refinement method in numerical analysis.

\subsection{Summary of Deep Learning algorithm}
Finally, we summarize our Deep Learning algorithm for the VPFP system as follows:
\begin{algorithm}[H]
\caption{Deep Learning algorithm for the VPFP system}\label{algorithm}
\begin{algorithmic}[1]
\For{number of epochs}
	\State \textbf{Sampling data:}
	   	\State\hskip1.5em Sample $m$ samples $t_1, t_2,..., t_m$ from [0,1] (or [0,5]).
	    \State\hskip1.5em Sample $n$ samples $x_1, x_2,..., x_n$ from [-1,1].
	    \State\hskip1.5em Sample $p$ samples $v_1, v_2,..., v_p$ from [-10,10].
	    \State\hskip1.5em Make a pair the samples to set the training data as \eqref{data_ge}, \eqref{data_ini} and \eqref{data_bdry}.
    	\State\hskip1.5em Add new top-$k$ training data paired with the velocity samples.
    \State \textbf{Evaluate the loss function:}
    	\State\hskip1.5em  Approximate the derivative of the DNN output (Autograd).
    	\State\hskip1.5em  Approximate the integration of the DNN output (Trapezoidal rule).
    	\State\hskip1.5em  Evaluate the loss function for the VPFP system \eqref{VPFP_loss_total}.
	\State \textbf{Updating parameters:}
    	\State\hskip1.5em  Update neural network parameters using the Adam optimizer:
	    \begin{align*}
	        w&\leftarrow w^{\text{new}},\\
	        b&\leftarrow b^{\text{new}},
	    \end{align*}
	    \State\hskip1.5em in the direction of minimizing the pre-defined loss function.
    \State \textbf{Grid Reuse technique:}
    	\State\hskip1.5em Save top$-k$ grid points $\{t_{\alpha_i},x_{\beta_i}\}_{i=1}^k$ which has the largest loss value.
\EndFor
\end{algorithmic}
\end{algorithm}

We also apply a similar Deep Learning algorithm to the PNP system.

\section{Part II. On convergence of DNN solutions to an analytic solution to the VPFP system and simulation results}\label{sec:part2}
In this section, we provide a DNN solution to the VPFP system. This section consists of three subsections. First, we propose the loss functions of the VPFP system for deep learning. We also prove the convergence of DNN solutions to an analytic solution of the VPFP system in two steps. Finally, we show that the simulation results on DNN solutions to the VPFP system agree with theoretical results by comparing the time-asymptotic behaviors and the macroscopic physical quantities which are defined in Section \ref{subsubsec:VPFP_equilibrium_quantities}.

We will focus on the VPFP system \eqref{VPFP} when the Knudsen number $\varepsilon$ is 1 in this section. The fixed Knudsen number can be arbitrarily chosen. For the sake of simplicity, we abuse notations and write $f_\varepsilon(t,x,v)$ as $f(t,x,v)$ and $E_\varepsilon(t,x)$ as $E(t,x)$ in this section. Later, in Part IV (Section \ref{sec:part4}), we consider the varied Knudsen number regimes.

\subsection{Loss functions for the VPFP system}
\label{subsec:part2_loss}
In Algorithm \ref{algorithm}, the Adam optimizer finds the optimal parameters $w^{\text{new}}$ and $b^{\text{new}}$ in the direction of minimizing a loss function. Thus, we need to define the loss functions for the Vlasov-Poisson-Fokker-Planck system: $Loss^{fp}_{GE}$ for the VPFP system \eqref{VPFP}$_1$ and \eqref{VPFP}$_3$, $Loss^{fp}_{IC}$ for the initial condition \eqref{VPFP}$_2$ and \eqref{VPFP}$_4$, $Loss^{fp}_{BC}$ for the boundary conditions \eqref{specular} and \eqref{VPFP}$_5$. Note that we use the superscript $Loss^{fp}$ for all loss functions to the VPFP system to distinguish it from the superscript $Loss^{pnp}$ used for the loss functions to the PNP system in Section \ref{subsec:part3_loss}.

First, we define the following loss functions for the governing equation \eqref{VPFP} as
\begin{multline}
  Loss^{fp}_{GE^{(1)}}(f^{nn})\\
  \eqdef \int_{(0,T)}dt\int_{(-1,1)}dx\int_{V}dv |\partial_t f^{nn}(t,x,v;m,w,b)+v \partial_x f^{nn}(t,x,v;m,w,b)\\
  +E^{nn}\partial_{v}f^{nn} - (\partial_{vv}f^{nn}(t,x,v;m,w,b) + \partial_v (vf^{nn})(t,x,v;m,w,b))|^2,
\end{multline}
and
\begin{multline}
  Loss^{fp}_{GE^{(2)}}(f^{nn})\\
  \eqdef \int_{(0,T)}dt\int_{(-1,1)}dx |\partial_x E^{nn}(t,x;m,w,b)-\int_V dv \ f^{nn}(t,x,v;m,w,b)|^2,
\end{multline}
where $V\eqdef [-10,10].$ Then we define $Loss^{fp}_{GE}$ as
\begin{equation}\label{VPFP_loss_ge}
Loss^{fp}_{GE}(f^{nn}) \eqdef Loss^{fp}_{GE^{(1)}} + Loss^{fp}_{GE^{(2)}}
\end{equation}
We now define the loss function for the initial condition via the use of the initial grid points as
\begin{equation}
Loss^{fp}_{IC^{(1)}}(f^{nn}) \eqdef \int_{(-1,1)}dx\int_{V} dv\left|f^{nn}(0,x,v)-f_0(x,v)\right|^2,
\end{equation}
and
\begin{multline}
Loss^{fp}_{IC^{(2)}}(f^{nn})\\
\eqdef \int_{(-1,1)}dx\left|E^{nn}(0,x;m,w,b) - \left(\int_{-1}^x dy\int_{\mathbb{R}}dvf_0(y,v)-(x+1)\right)\right|^2.
\end{multline}
Note that we use the equation \eqref{VPFP_E_ini} for the loss function $Loss_{IC}^2$. Then, we define $Loss_{IC}$ as 
$$
  Loss^{fp}_{IC}(f^{nn}) \eqdef Loss^{fp}_{IC^{(1)}} + Loss^{fp}_{IC^{(2)}}.
$$
The loss functions for the \textit{specular} boundary condition for $f$ in Section \ref{subsec:boundary_VPFP} and the Dirichlet boundary condition for $E$ \eqref{VPFP}$_5$ are defined as
\begin{multline}
Loss^{fp}_{BC^{(1)}}(f^{nn})
\\ \eqdef \int_{(0,T)}dt\int_{\gamma_-} dxdv\left|f^{nn}(t,x,v;m,w,b)-f^{nn}(t,x,-v;m,w,b)\right|^2,
\end{multline}
and
\begin{equation}
Loss^{fp}_{BC^{(2)}}(f^{nn})\eqdef\int_{(0,T)}dt\sum_{x\in\{-1,1\}}|E^{nn}(t,x;m,w,b)|^2.
\end{equation}
Then we define the total loss for the boundary conditions as
$$
Loss^{fp}_{BC}(f^{nn}) \eqdef Loss^{fp}_{BC^{(1)}} + Loss^{fp}_{BC^{(2)}}.
$$
Finally, we define the total loss as
\begin{equation}\label{VPFP_loss_total}
  Loss^{fp}_{Total}(f^{nn}) \eqdef Loss^{fp}_{GE} + Loss^{fp}_{IC} + Loss^{fp}_{BC}.
\end{equation}
Note that we compute these loss functions via the approximation of the integration by the Riemann sum on the grid points, which is explained in Section \ref{subsec:grid}. For example, the loss function $Loss^{fp}_{GE}(f^{nn})$ can be approximated as
\begin{multline}
  Loss^{fp}_{GE}\approx \frac{1}{N_{i,j,k}}\sum_{i,j,k} \bigg|\partial_t f^{nn}(t_i,x_j,v_k;m,w,b) +v \partial_x f^{nn}(t_i,x_j,v_k;m,w,b)\\
  +E^{nn}\partial_{v}f^{nn} - (\sigma\partial_{vv}f^{nn}(t_i,x_j,v_k;m,w,b) + \beta \partial_v (vf^{nn})(t_i,x_j,v_k;m,w,b))\bigg|^2 \\
  + \frac{1}{N_{i,j}}\sum_{i,j}\bigg|\partial_x E^{nn}(t_i,x_j;m,w,b)-\int_V dv \ f^{nn}(t_i,x_j,v_k;m,w,b)\bigg|^2,
\end{multline}
where $N_{i,j,k}$ and $N_{i,j}$ are the number of grid points.

\subsection{On convergence of DNN solutions to analytic solutions to the VPFP system}\label{subsec:part2_theory}
In this section, we show the convergence of DNN solutions to analytic solutions to the VPFP system \eqref{VPFP} in two steps. We first prove that there exists a sequence of neural network parameters (neuron numbers $m$, weights $w$ and biases $b$ as defined in Section \ref{subsec:method}) such that the total loss function $Loss^{fp}_{Total}$ converges to 0. Sequentially, we also prove that if we minimize the total loss function $Loss^{fp}_{Total}$, it implies that the Deep Neural Network solution converges to an analytic solution. Throughout the section, we assume that the existence and the uniqueness of solutions for the VPFP system \eqref{VPFP} with the specular boundary condition \eqref{specular} are a priori given.

We first introduce the following definition and the theorem from \cite{li1996simultaneous} on the existence of approximated neural network solutions:
\begin{definition}[Li, \cite{li1996simultaneous}]\label{C_hat} For a compact set $K$ of $\mathbb{R}^n$, we say $f\in \widehat{C}^m(K)$, $m\in \mathbb{Z}_+^n$ if there is an open $\Omega$ (depending on $f$) such that $K\subset \Omega$ and $f\in C^m(\Omega).$
\end{definition}

\begin{theorem}[Li, Theorem 2.1, \cite{li1996simultaneous}]\label{VPFP_global} Let $K$ be a compact subset of $\mathbb{R}^n$, $n\ge 1$, and $f\in\widehat{C}^{m_1}(K)\cap\widehat{C}^{m_2}(K)\cap \cdots \widehat{C}^{m_q}(K)$, where $m_i \in \mathbb{Z}^n_+$ for $1\le i\le q$. Also, let $\bar{\sigma}$ be any non-polynomial function in $C^l(\mathbb{R})$, where $l=\max\{|m_i|:1\le i\le q\}$. Then for any $\epsilon>0,$ there is a network
$$f^{nn}(x)=\sum_{i=0}^\nu c_i\bar{\sigma}(\langle w_i,x\rangle +b_i), \ x\in \mathbb{R}^n,$$ where $c_i\in \mathbb{R},$ $w_i\in \mathbb{R}^n$, and $b_i\in \mathbb{R}$, $0\le i\le \nu$ such that 
$$\|D^kf-D^kf^{nn}\|_{L^{\infty}(K)}<\epsilon,$$
for $k\in \mathbb{Z}^n_+$, $k\le m_i$, for some $i$, $1\le i\le q.$
\end{theorem}
\begin{remark} We can generalize the result above to the one with several hidden layers (see, \cite{hornik1989multilayer}). Also, we may assume that the architecture is assumed to have only one hidden layer; i.e., $L=2$. 
\end{remark}

Now we introduce our first main theorem which states that a sequence of neural network solutions that makes the total loss function converge to zero exists if a $\widehat{C}^{(1,1,2)}$ solution to the VPFP system exists:
\begin{theorem}[Theorem 3.4 of \cite{MR4116803}]\label{thm:vpfp_loss_0}Assume that the number of layers $L=2$ and that the solution $f$ to \eqref{VPFP} with \eqref{specular} which belongs to $\widehat{C}^{(1,1,2)}([0,T]\times[-1,1]\times V)$, and the activation function $\bar{\sigma}(x)\in C^{(2,2,3)}([0,T]\times [-1,1]\times V)$ is non-polynomial. Then, there exists $\{m_{[j]}, w_{[j]}, b_{[j]}\}_{j=1}^\infty$ such that a sequence of the DNN solutions $f^{nn}$ of Theorem \ref{VPFP_global} with $m_{[j]}$ nodes, denoted by $$\{f_j(t,x,v) = f^{nn}(t,x,v;m_{[j]}, w_{[j]}, b_{[j]})\}_{j=1}^{\infty}$$ satisfies\footnote{Each of $m_{[j]}, w_{[j]}, b_{[j]}$ represents the matrix of the numbers corresponding to $f_j$ for each $j=1,2,...,\infty$. The matrices $m_{[j]}, w_{[j]}, b_{[j]}$ consist of the element represented as $m_{[j],ik}^{(l)}, w_{[j],ik}^{(l)}, b_{[j],ik}^{(l)}$, respectively.}
\begin{equation}\label{VPFP_loss_0}
Loss^{fp}_{Total}(f_j)\rightarrow 0\text{ as }j\rightarrow\infty.
\end{equation}
\end{theorem}

\begin{proof}This is similar to that of Theorem 3.4 of \cite{MR4116803}.
\end{proof}
\begin{remark}
The assumption $f\in\widehat{C}^{(1,1,2)}([0,T]\times[-1,1]\times V)$ can be replaced by a general Sobolev space, since the functions in a Sobolev space can be approximated by continuous functions on a compact set.
\end{remark}
The first main theorem \ref{thm:vpfp_loss_0} provides us that we can find the neural network parameters that reduce the pre-defined total loss function as much as we want. However, it does not imply that the DNN solutions converge to an analytic solution to the VPFP system. Therefore, we introduce our second main theorem, Theorem \ref{thm:vpfp_converge}, which shows that the DNN solutions converge to an analytic solution in a suitable function space when we minimize the total loss function $Loss^{fp}_{Total}$. We assume that our compact domain $V=[-10,10]$ of the $v$-variable is chosen sufficiently large so that we can have
\begin{equation}\label{VPFP_artificial_bdry}
||f||_{L^1_x([-1,1];L^1_v(\mathbb{R}\setminus V))}\le \epsilon\ \text{and }
\left|\partial^{k}_{v}f(t,x,v)-\partial^{k}_{v}f^{nn}(t,x,v)\right|_{v\in \partial V}\leq \epsilon,
\end{equation}for some sufficiently small $\epsilon>0$ and $k=0,1$.
\begin{theorem}\label{thm:vpfp_converge} Assume that $f$ is a solution to \eqref{VPFP} with \eqref{specular} which belongs to $\widehat{C}^{(1,1,2)}([0,T]\times[-1,1]\times V)$. If the solution $f$ and the Deep Neural Network solution $f^{nn}(t,x,v;m,w,b)$ satisfy \eqref{VPFP_artificial_bdry}, then it implies that
\begin{equation}
\left\| f^{nn}(\cdot,\cdot,\cdot;m,w,b) - f \right\|_{L_t^{\infty}([0,T];L_{x,v}^{2}([-1,1]\times V))}\leq C(Loss^{fp}_{Total}(f^{nn})+\epsilon),
\end{equation}
where $C$ is a positive constant depending only on $T$.
\end{theorem}
\begin{remark}\label{remark:VPFP_theory_simulation}
Note that we fix the DNN architecture in Figure \ref{fig:DNN_structure_VPFP} before we train the DNN. Namely, we first fix the number of neurons for each layer $m$ before training, and then we update the weights $w$ and biases $b$ to minimize the total loss function. Therefore, if we want to approximate the DNN solution $f^{nn}$ to an analytic solution to the VPFP system, Theorem \ref{thm:vpfp_converge} indicates how much the total loss function $Loss_{Total}(f^{nn}(t,x,v;m,w,b))$ has to be reduced. Then, Theorem \ref{thm:vpfp_loss_0} guarantees the existence of a 3-tuple ($m$, $w$, $b$) where the total loss function is sufficiently reduced as we want. In the DNN simulation, we use Algorithm \ref{algorithm} to find the optimal weights $w$ and biases $b$ to reduce the total loss function while the number of neurons for each layer $m$ is fixed.
\end{remark}
\begin{proof}Motivated by \cite{MR1634851}, we define a transform $\bar{u}(t,x,v)$ of a function $u(t,x,v)$ as follows:
$$\bar{u}(t,x,v) = e^{- t}u(t,x,e^{- t}v).$$
Then the transformed function $\bar{f}$ satisfies
$$\partial_t\bar{f} +e^{- t}(v\cdot\partial_x)\bar{f}+e^{ t}E\partial_v\bar{f} - e^{2 t}\partial_{v}^{2}\bar{f} =0.$$
Also, we define the error values of the functions $\bar{f}^{nn}$ and $E^{nn}$ as the following equations:
$$d_{ge,j}^{(1)}(t,x,v)\eqdef -\left[\partial_{t}+e^{- t}(v\cdot\partial_{x})+e^{ t}E^{nn}\partial_v- e^{2 t}\partial_{v}^{2} \right]\bar{f}^{nn}$$
for $(t,x,v)\in[0,T]\times[-1,1]\times e^{ t}V$,
$$d_{ge,j}^{(2)}(t,x)\eqdef -(\partial_xE^{nn} - \int dv \bar{f}^{nn}(t,x,v)+h(x))$$
for $(t,x)\in[0,T]\times[-1,1]$,
$$d_{bc,j}^{(1)}(t,x,v)\eqdef-(\bar{f}^{nn}(t,x,v)-\bar{f}^{nn}(t,x,-v))$$
for $(t,x,v)\in \gamma^{-}_{T,e^{ t}V}$, and
$$d_{bc,j}^{(2)}(t,x)\eqdef -E^{nn}(t,x)$$
for $(t,x)\in [0,T]\times\partial[-1,1]$. Note that the interval $e^{ t}V$ is defined as
$$ e^{ t}V \eqdef  [-10e^{ t}, 10e^{ t}],$$
and $\gamma^{\pm}_{T,e^{ t}V}$ is defined as $[0,T]\times \gamma^\pm_{e^{ t}V}$, where $\gamma^\pm_{e^{ t}V}$ is equal to $\gamma^\pm$ with the velocity domain $\mathbb{R}$ is replaced by $e^{ t}V$.

We now consider the following equation on the difference between $\bar{f}$ and $\bar{f}^{nn}$ for each fixed $j$ on the compact set of $t,x,v$ only as 
\begin{equation}\label{VPFP_diff_ge}
\left[\partial_{t}+e^{- t}(v\cdot\partial_{x})- e^{2 t}\partial_{v}^{2} \right]\{\bar{f}-\bar{f}^{nn}\}+e^{ t}(E\partial_v\bar{f}-E^{nn}\partial_v\bar{f}^{nn})= d_{ge,j}^{(1)}(t,x,v).
\end{equation}

Then we derive the inequality below by multiplying $2(\bar{f}-\bar{f}^{nn})$ onto \eqref{VPFP_diff_ge} and integrating it over $[-1,1]\times e^{ t}V$ as
\begin{multline}\label{VPFP_energy_eq}
\int_{-1}^{1}\int_{-10e^{ t}}^{10e^{ t}}\frac{\partial}{\partial t}(\bar{f}-\bar{f}^{nn})^{2}(t,x,v)dvdx+ \left(\int_{\gamma_{e^{ t}V}}(\bar{f}-\bar{f}^{nn})^{2}d\gamma\right)\\
- 2 e^{2 t}\langle\partial_{v}^{2}(\bar{f}-\bar{f}^{nn}),(\bar{f}-\bar{f}^{nn})\rangle\\
= -\int_{-1}^{1}\int_{-10e^{ t}}^{10e^{ t}}2e^{ t} (E\partial_v\bar{f}-E^{nn}\partial_v\bar{f}^{nn})(\bar{f}-\bar{f}^{nn})dvdx+2\langle d_{ge,j}^{(1)},(\bar{f}-\bar{f}^{nn})\rangle, 
\end{multline}
where $\langle\cdot,\cdot\rangle$ denotes the standard inner product on ${L^{2}([-1,1]\times e^{ t}V)}$. On the left-hand side of \eqref{VPFP_energy_eq}, we note that
\begin{multline}
\int_{\gamma_{e^{ t}V}}(\bar{f}-\bar{f}^{nn})^{2}d\gamma
= \int_{\gamma^{+}_{e^{ t}V}}(\bar{f}-\bar{f}^{nn})^{2}d\gamma -\int_{\gamma^{-}_{e^{ t}V}}(\bar{f}-\bar{f}^{nn})^{2}d\gamma
\\= \int_{\gamma^{+}_{e^{ t}V}}(\bar{f}-\bar{f}^{nn})^{2}d\gamma -\int_{\gamma^{-}_{e^{ t}V}}(d_{bc,j}^{(1)}(t,x,v) + (\bar{f}-\bar{f}^{nn})(t,x,-v))^{2}d\gamma
\\\geq \int_{\gamma^{+}_{e^{ t}V}}(\bar{f}-\bar{f}^{nn})^{2}d\gamma -2\int_{\gamma^{-}_{e^{ t}V}}|d_{bc,j}^{(1)}(t,x,v)|^2d\gamma - 2\int_{\gamma^{-}_{e^{ t}V}} (\bar{f}-\bar{f}^{nn})^2(t,x,-v)d\gamma
\\= 3\int_{\gamma^{+}_{e^{ t}V}}(\bar{f}-\bar{f}^{nn})^{2}d\gamma -2\int_{\gamma^{-}_{e^{ t}V}}|d_{bc,j}^{(1)}(t,x,v)|^2d\gamma \geq -2\int_{\gamma^{-}_{e^{ t}V}}|d_{bc,j}^{(1)}(t,x,v)|^2d\gamma.
\end{multline}
Also, note that
\begin{multline}\nonumber
\frac{d}{dt}\left\|(\bar{f}-\bar{f}^{nn})(t,\cdot,\cdot)\right\|^{2}_{L^{2}_{x,v}([-1,1]\times e^{ t}V)} = \int_{-1}^{1}\int_{-10e^{ t}}^{10e^{ t}}\frac{\partial}{\partial t}(\bar{f}-\bar{f}^{nn})^{2}(t,x,v)dvdx \\ + \underbrace{10 e^{ t}\left(\left\|(\bar{f}-\bar{f}^{nn})(t,\cdot,10e^{ t})\right\|^{2}_{L^{2}_{x}([-1,1])} + \left\|(\bar{f}-\bar{f}^{nn})(t,\cdot,-10e^{ t})\right\|^{2}_{L^{2}_{x}([-1,1])}\right)}_{\eqdef B_{1}(t)},
\end{multline}
by the Leibniz rule and
\begin{multline}\nonumber
2 e^{2 t}\langle\partial_{v}^{2}(\bar{f}-\bar{f}^{nn}),(\bar{f}-\bar{f}^{nn})\rangle = -2 e^{2 t} \left\|\partial_{v}(\bar{f}-\bar{f}^{nn}) \right\|^{2}_{L^{2}_{x,v}([-1,1]\times e^{ t}V)}
\\+ \underbrace{2 e^{2 t}\int_{-1}^{1}\partial_{v}(\bar{f}-\bar{f}^{nn})(\bar{f}-\bar{f}^{nn})(t,\cdot,10e^{ t})-\partial_{v}(\bar{f}-\bar{f}^{nn})(\bar{f}-\bar{f}^{nn})(t,\cdot,-10e^{ t})dx}_{\eqdef B_2(t)}.
\end{multline}
So, it yields that
$$2 e^{2 t}\langle\partial_{v}^{2}(\bar{f}-\bar{f}^{nn}),(\bar{f}-\bar{f}^{nn})\rangle \leq B_2(t).$$
Also, note that $$(E\partial_v\bar{f}-E^{nn}\partial_v\bar{f}^{nn})(\bar{f}-\bar{f}^{nn})
=(E-E^{nn})\partial_v\bar{f}(\bar{f}-\bar{f}^{nn})+\frac{1}{2}E^{nn}\partial_v(\bar{f}-\bar{f}^{nn})^2,$$
and Poisson equation implies the difference between $E$ and $E^{nn}$ as 
\begin{equation}
	\partial_x(E-E^{nn}) =\int_{e^{ t}V} (\bar{f}-\bar{f}^{nn}) dv + \int_{\R\setminus e^{ t}V} \bar{f} dv + d_{ge,j}^{(2)}(t,x).
\end{equation}
So, it yields that
\begin{multline*}|E-E^{nn}|(t,x)\\
\le \int_{-1}^x dx' \int_{e^{ t}V}dv |\bar{f}-\bar{f}^{nn}|(t,x',v)+\int_{-1}^x dx' \left(\int_{\mathbb{R}\setminus e^{ t}V} |\bar{f}(t,x',v)|dv \right)\\
+\int_{-1}^x dx' \ |d_{ge,j}^{(2)}(t,x')|+|\underbrace{(E-E^{nn})(t,x=-1)}_{=|E^{nn}(t,x=-1)|}|\\
\le \int_{-1}^x dx' \int_{e^{ t}V}dv |\bar{f}-\bar{f}^{nn}|(t,x',v)+\int_{-1}^x dx' \left(\int_{\mathbb{R}\setminus V} dv|f(t,x',v)| \right)
\\+\sqrt{2}\| d_{ge,j}^{(2)}(t,x') \|_{L^2_{x}([-1,1])}+\underbrace{\|d_{bc,j}^{(2)}(t,x)\|_{L^1_{x}(\partial[-1,1])}}_{\leq \sqrt{2}\|d_{bc,j}^{(2)}(t,x)\|_{L^2_{x}(\partial[-1,1])}}
,\end{multline*} by H\"older's inequality. Thus, we have
\begin{multline}\label{mathcalL}\|E(t)-E^{nn}(t)\|_{L^\infty_x}
\\\le \|\bar{f}-\bar{f}^{nn}\|_{L^1_{x,v}([-1,1]\times e^{ t}V)} + \epsilon+\sqrt{2}\| d_{ge,j}^{(2)}\|_{L^2_{x}}+\sqrt{2}\|d_{bc,j}^{(2)}\|_{L^2_{x}},\end{multline} by \eqref{VPFP_artificial_bdry}.
Then the integration by parts in $v$ variable yields that
\begin{multline}
\left|\int_{-1}^{1}\int_{-10e^{ t}}^{10e^{ t}}2e^{ t} (E\partial_v\bar{f}-E^{nn}\partial_v\bar{f}^{nn})(\bar{f}-\bar{f}^{nn})dvdx\right|\\
\le\int_{-1}^{1}\int_{-10e^{ t}}^{10e^{ t}}2e^{ t} |E-E^{nn}||\partial_v\bar{f}||\bar{f}-\bar{f}^{nn}|dvdx\\+\frac{1}{2}\left|\int_{-1}^{1}\int_{-10e^{ t}}^{10e^{ t}}2e^{ t}E^{nn}\partial_v(\bar{f}-\bar{f}^{nn})^2dvdx\right|\\
\le 2e^{ t} \bigg( \|\bar{f}-\bar{f}^{nn}\|_{L^1_{x,v}([-1,1]\times e^{ t}V)} + \epsilon+\sqrt{2}\| d_{ge,j}^{(2)}\|_{L^2_{x}}+\sqrt{2}\|d_{bc,j}^{(2)}\|_{L^2_{x}}\bigg)
\\\times \|\bar{f}\|_{C^1_{x,v}([-1,1]\times e^{ t}V)}\|\bar{f}-\bar{f}^{nn}\|_{L^1_{x,v}([-1,1]\times e^{ t}V)}
\\+\frac{1}{2}\left|\int_{-1}^{1}2e^{ t}E^{nn}\left((\bar{f}-\bar{f}^{nn})^2(t,x,10e^{ t})-(\bar{f}-\bar{f}^{nn})^2(t,x,-10e^{ t})\right)dx\right|\\
\le 2e^{ t} \bigg( \|\bar{f}-\bar{f}^{nn}\|_{L^1_{x,v}([-1,1]\times e^{ t}V)} + \epsilon+\sqrt{2}\| d_{ge,j}^{(2)}\|_{L^2_{x}}+\sqrt{2}\|d_{bc,j}^{(2)}\|_{L^2_{x}}\bigg)
\\\times \|\bar{f}\|_{C^1_{x,v}([-1,1]\times e^{ t}V)}\|\|\bar{f}-\bar{f}^{nn}\|_{L^1_{x,v}([-1,1]\times e^{ t}V)}\\+e^{ t}\|E^{nn}\|_{L^1_x} \left(\|(\bar{f}-\bar{f}^{nn})(t,\cdot,10e^{ t})\|^2_{L^\infty_x}+\|(\bar{f}-\bar{f}^{nn})(t,\cdot,-10e^{ t})\|^2_{L^\infty_x}\right)
\end{multline}Here, by H\"older's inequality, we have
$$\|\bar{f}-\bar{f}^{nn}\|_{L^1_{x,v}([-1,1]\times e^{ t}V)}\le\sqrt{40e^{ t}}\|\bar{f}-\bar{f}^{nn}\|_{L^2_{x,v}} .$$ Also, $f\in \hat{C}^{(1,1,2)}$ implies that
$$ \|\bar{f}\|_{C^1_{x,v}([-1,1]\times e^{ t}V)}\|=\|e^{- t}f(t,x,e^{- t}v)\|_{C^1_{x,v}([-1,1]\times e^{ t}V)} \leq e^{- t}(C_0+e^{- t}C_0),$$
for some positive constant $C_0$. Thus, by \eqref{VPFP_artificial_bdry}, we have
\begin{multline}\label{eq18}
\left|\int_{-1}^{1}\int_{-10e^{ t}}^{10e^{ t}}2e^{ t} (E\partial_v\bar{f}-E^{nn}\partial_v\bar{f}^{nn})(\bar{f}-\bar{f}^{nn})dvdx\right|\\
\le  2\left(\sqrt{40e^{ t}}\|\bar{f}-\bar{f}^{nn}\|_{L^2_{x,v}} +\epsilon+\sqrt{2}\| d_{ge,j}^{(2)}\|_{L^2_{x}}+\sqrt{2}\|d_{bc,j}^{(2)}\|_{L^2_{x}}\right)
\\\times (C_0+e^{- t}C_0)\sqrt{40e^{ t}}\|\bar{f}-\bar{f}^{nn}\|_{L^2} + 2\|E^{nn}\|_{L^1_x}e^{  t}\epsilon.
\end{multline}
Also, by \eqref{mathcalL}, we have
\begin{multline*}
\|E^{nn}\|_{L^1_x}=\|E^{nn}-E+E\|_{L^1_x}
\le \|E-E^{nn}\|_{L^1_x}+\|E\|_{L^1_x}\\
\le 2(\|f-f_j\|_{L^1_{x,v}([-1,1]\times V)} + \epsilon+\sqrt{2}\| d_{ge,j}^{(2)}\|_{L^2_{x}}+\sqrt{2}\|d_{bc,j}^{(2)}\|_{L^2_{x}})+\|E\|_{L^1_x}\\
\le  2\left(\sqrt{40e^{ t}}\|\bar{f}-\bar{f}^{nn}\|_{L^2_{x,v}} + \epsilon+\sqrt{2}\| d_{ge,j}^{(2)}\|_{L^2_{x}}+\sqrt{2}\|d_{bc,j}^{(2)}\|_{L^2_{x}}\right)+\|E\|_{L^1_x},
\end{multline*} by \eqref{VPFP_artificial_bdry}. Therefore, \eqref{eq18} yields that if $\epsilon<1$, then
\begin{multline}\label{eq19}
\left|\int_{-1}^{1}\int_{-10e^{ t}}^{10e^{ t}}2e^{ t} (E\partial_v\bar{f}-E^{nn}\partial_v\bar{f}^{nn})(\bar{f}-\bar{f}^{nn})dvdx\right|
\\\le 2(C_0+e^{- t}C_0)\sqrt{40e^{ t}}\|\bar{f}-\bar{f}^{nn}\|_{L^2_{x,v}}
\\\times\bigg(\sqrt{40e^{ t}}\|\bar{f}-\bar{f}^{nn}\|_{L^2_{x,v}} +\epsilon+\sqrt{2}\| d_{ge,j}^{(2)}\|_{L^2_{x}}+\sqrt{2}\|d_{bc,j}^{(2)}\|_{L^2_{x}}\bigg)
\\+2e^{  t}\epsilon\left(\left(\sqrt{40e^{ t}}\|\bar{f}-\bar{f}^{nn}\|_{L^2_{x,v}}+ \epsilon+\mathcal{L}(t)\right)+\|E\|_{L^1_x} \right)\\
\le C_1\|\bar{f}-\bar{f}^{nn}\|^2_{L^2_{x,v}}+C_2\epsilon+C_3\| d_{ge,j}^{(2)}\|_{L^2_{x}}^2+C_4\|d_{bc,j}^{(2)}\|_{L^2_{x}}^2,
\end{multline}for some positive constant $C_1$, $C_2$, $C_3$, and $C_4$ by the use of Young's inequality. Note that $\|E\|_{L^1_x}$ is also bounded as we have $f\in \hat{C}^{(1,1,2)}$ on the compact domain and we also have $\|f\|_{L^1_x([-1,1];L^1_v (\mathbb{R}\setminus V))}\le \epsilon$ by \eqref{VPFP_artificial_bdry}.
Then, we can reduce \eqref{VPFP_energy_eq} to
\begin{multline}\label{VPFP_energy_ineq}
\frac{d}{dt}\overbrace{\|\bar{f}-\bar{f}^{nn}\|^{2}_{L^{2}([-1,1]\times e^{ t}V)}}^{Y(t)\eqdef }
\\ \leq (C_1+1)\left\|\bar{f}-\bar{f}^{nn}\right\|^2_{L^2_{x,v}([-1,1]\times e^{ t}V)} + B_{1}(t) + B_{2}(t) + C_4\epsilon
\\+\|d_{ge,j}^{(1)}(t,x,v)\|_{L^{2}_{x,v}([-1,1]\times e^{ t}V)}^2+ C_3\| d_{ge,j}^{(2)}(t,x)\|_{L^2_{x}}^2
\\+2\int_{\gamma^{-}_{T,e^{ t}V}}|d_{bc,j}^{(1)}(t,x,v)|^2d\gamma+C_4\|d_{bc,j}^{(2)}(t,x)\|_{L^2_{x}}^2.
\end{multline}
If we use $C_1+1\leq C_5$ with some positive constant $C_5$, we can rewrite \eqref{VPFP_energy_ineq} as follows:
\begin{multline}\label{VPFP_energy_ineq2}
Y'(t) -C_5Y(t) \leq B_{1}(t) + B_{2}(t)+C_4\epsilon
\\+C_6\underbrace{\left(\|d_{ge,j}^{(1)}\|_{L^{2}_{x,v}}^2 + \| d_{ge,j}^{(2)}\|_{L^2_{x}}^2 + \int_{\gamma^{-}_{T,e^{ t}V}}|d_{bc,j}^{(1)}|^2d\gamma + \|d_{bc,j}^{(2)}\|_{L^2_{x}}^2\right)}_{\eqdef L(t)},
\end{multline}
with some positive constant $C_6$. By Gr\"onwall's inequality, we have
\begin{multline}\label{VPFP_conclusion}
Y(t)\leq e^{C_5t}
\\\left(Y(0)+\int_{0}^{t} e^{-C_5s}C_6L(s)ds+\int_{0}^{t} e^{-C_5s}\left(B_1(s)+B_2(s)+C_4\epsilon\right) ds\right).
\end{multline}
Finally, we recall that $Y(t)=e^{- t}\|f-f_j\|^{2}_{L^{2}_{x,v}([-1,1]\times V)}$, $Y(0)=Loss^{fp}_{IC}$, and
\begin{multline}\nonumber
Y(0)+\int_{0}^{t}e^{-C_5s}L(s)ds \leq Y(0)+\int_{0}^{t}L(s)ds \leq Y(0)+\int_{0}^{T}L(s)ds
\\\leq C_7(Loss^{fp}_{IC} + Loss^{fp}_{ge} + Loss^{fp}_{BC}) = C_7Loss^{fp}_{Total}(f_{j}).
\end{multline}
for some positive constant $C_7$. Moreover, under the assumption on \eqref{VPFP_artificial_bdry}, we have
$$\int_{0}^{t}e^{-C_5s}B_{1}(s)ds\leq 40\epsilon^{2}\int_{0}^{t}e^{-C_5s}e^{- s}ds\leq C_{8} \epsilon^{2},$$
$$
\int_{0}^{t}e^{-C_5s}B_{2}(s)ds\leq 8\epsilon^{2} \int_{0}^{t}e^{-C_5s}e^{- s}ds \leq C_{9}\epsilon^{2},$$
for some positive constant $C_8$ and $C_9$. Therefore, \eqref{VPFP_conclusion} and the inverse transform from $\bar{f}$ to $f$ imply that
$$\left\| f-f^{nn}\right\|_{L_t^{\infty}([0,T];L_{x,v}^{2}([-1,1]\times V))}\leq C(Loss_{Total}(f^{nn}) + \epsilon),$$
for some positive constant $C$ which depends only on $T$. This completes the proof of Theorem \ref{thm:vpfp_converge}. 
\end{proof}

\subsection{Neural Network simulations}\label{subsec:part2_simulation}
In this section, we introduce numerical simulations for the solutions $f^{nn}(t,x,v;m,w,b)$ and $E^{nn}(t,x;m,w,b)$ to the VPFP system \eqref{VPFP}. We consider the following initial condition:
\begin{equation}\label{VPFP_initial}
f(0,x,v)=f_0(x,v) = \begin{cases}
e^{x-1}\left( 1-\cos(\frac{\pi}{2}v) \right), &\text{ if $v\in  (-4,4),$}\\
0, &\text{ otherwise},
\end{cases}
\end{equation}
which has different initial ditributions at each position $x\in[-1,1]$. We consider the time interval $[0,5]$ which is enough to reach the steady state of the solution to the VPFP system. Also, we set the background charge $h(x)$ as constant that safisfies the global neutrality condtion \eqref{VPFP_neutrality}. More details about our Deep Learning algorithm are explained in Section \ref{subsec:dnn_architerture}, Section \ref{subsec:grid}, and the summary of the Deep Learning Algorithm \ref{algorithm}.

The first plot in Figure \ref{fig:vpfp_1} shows the time-asymptotic behaviors of the $L^\infty$ norm of the distribution $f^{nn}(t,x,v;m,w,b)$ with respect to position $x$ and velocity $v$. After 3 time grids, the value converges to almost constant. This indicates that the distributions $f^{nn}(t,x,v;m,w,b)$ converge to the steady state. It can be observed more clearly in the third plot in Figure \ref{fig:vpfp_1}, which shows the difference between the distribution $f^{nn}(t,x,v;m,w,b)$ and the global equilibrium \eqref{equilibrium_VPFP}. The $L^1$, $L^2$ and $L^\infty$ norm of the difference with respect to position $x$ and velocity $v$ tend to zero as time increases. This is consistent to our theoretical supports provided in the equation \eqref{equilibrium_VPFP}. Later, the pointwise values of $f^{nn}(t,x,v;m,w,b)$ show the shape of the convergence to the global Maxwellian in \ref{fig:vpfp_4}.

The second plot in Figure \ref{fig:vpfp_1} shows the value of $\text{Mass}(t)$ over time defined in \eqref{mass conservation}. The plot shows that the total mass of the system is conserved. It shows an agreement with the theoretical result that the VPFP system with the specular boundary condition \eqref{specular} yields the conservation of the total mass \eqref{mass conservation}, which is an important a \text{priori} physical law for the VPFP system.

\begin{figure}[H]
  \includegraphics[width=\textwidth, draft=false]{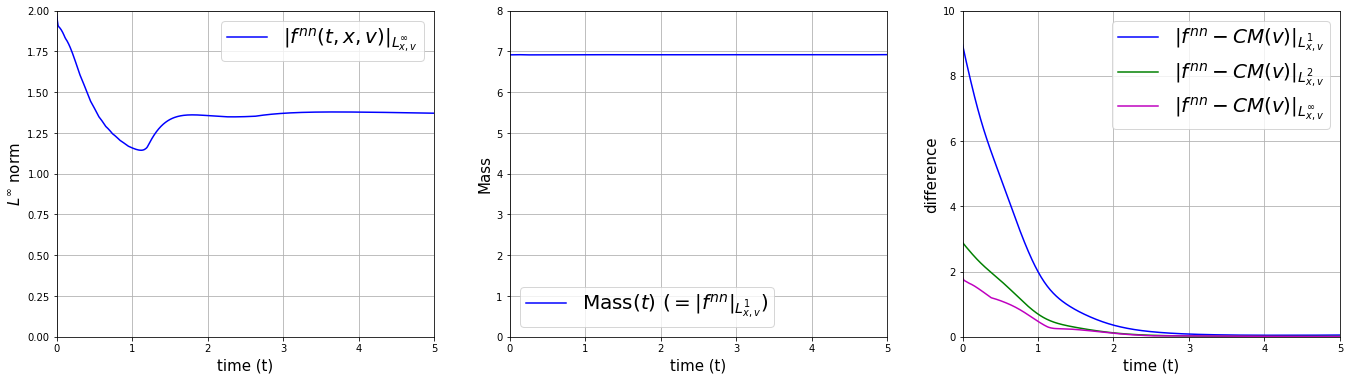}
  \caption{The time-asymptotic behaviors of the $L^\infty$ norm, $L^1$ norm of $f^{nn}(t,x,v;m,w,b)$ (the first and the second plot) and the $L^1$ norm, $L^2$ norm, and $L^\infty$ norm of the difference between $f^{nn}(t,x,v;m,w,b)$ and the global Maxwellian $\frac{\|f_0(\cdot,\cdot)\|_{L^1_{x,v}}}{|\Omega|}M(v)$. It is notable that the total mass $Mass(t)$ of the distribution $f^{nn}(t,x,v;m,w,b)$ is conserved over time in the second plot. Also, note that the third plot shows that the distribtuion $f^{nn}(t,x,v;m,w,b)$ converges to the global Maxwellian.}
  \label{fig:vpfp_1}
\end{figure}

\begin{figure}[H]
  \includegraphics[width=\textwidth, draft=false]{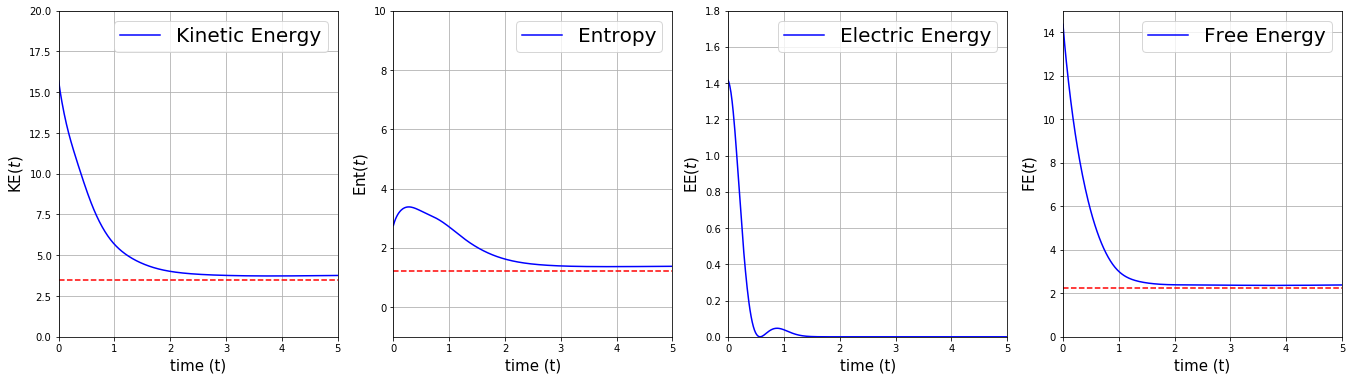}
  \caption{The time-asymptotic behaviors of the macroscopic quantities of $f^{nn}(t,x,v;m,w,b)$ and $E^{nn}(t,x;m,w,b)$. The steady-state values of the kinetic energy \eqref{KE_inf}, the entropy \eqref{Ent_inf}, the free energy \eqref{FE_inf} are indicated in the red-dotted lines. Note that the free energy is monotonically decreasing.}
  \label{fig:vpfp_2}
\end{figure}

Figure \ref{fig:vpfp_2} shows the time-asymptotic behaviors of four macroscopic quantities of $f^{nn}(t,x,v;m,w,b)$; the total kinetic energy ``\text{KE}" \eqref{KE}, the entropy ``\text{Ent}" \eqref{Ent}, the electric pontential energy ``\text{EE}" \eqref{EE} and the free energy ``\text{FE}" \eqref{FE}. The steady state values of these four macroscopic quantities can obtained from the macroscopic quntities of the equilibrium in \eqref{equilibrium_VPFP}. Therefore, we expect the steady state values of the four macroscopic qunatities as follows:
\begin{align}
  \text{KE}_\infty&= \frac{\|f_0(\cdot,\cdot)\|_{L^1_{x,v}}}{|\Omega|}, \label{KE_inf}\\
  \text {Ent}_\infty&= -\|f_0(\cdot,\cdot)\|_{L^1_{x,v}}\log\left(\frac{\|f_0(\cdot,\cdot)\|_{L^1_{x,v}}}{|\Omega|(2\pi)^{0.5}}\right)+\frac{1}{2}\|f_0(\cdot,\cdot)\|_{L^1_{x,v}}, \label{Ent_inf}\\
  \text {EE}_\infty&=0, \label{EE_inf}\\
  \text{FE}_{\infty}&= \text{KE}_{\infty} -\text{Ent}_{\infty} + \text{EE}_{\infty}, \label{FE_inf}
\end{align}
where $|\Omega|=2$ and $\|f_0(\cdot,\cdot)\|_{L^1_{x,v}}\approx6.917$ in our case.  We denote the steady state values via the red-dotted lines in Figure \ref{fig:vpfp_2}. The four plots show that the each physical quantitity converges to each steady state. Also, the fourth plot in Figure \ref{fig:vpfp_2} shows a non-increasing trend of the free energy. This is also consistent to our theoretical supports of \eqref{FE_dissipation}.

\begin{figure}[H]
  \includegraphics[width=0.9\textwidth, draft=false]{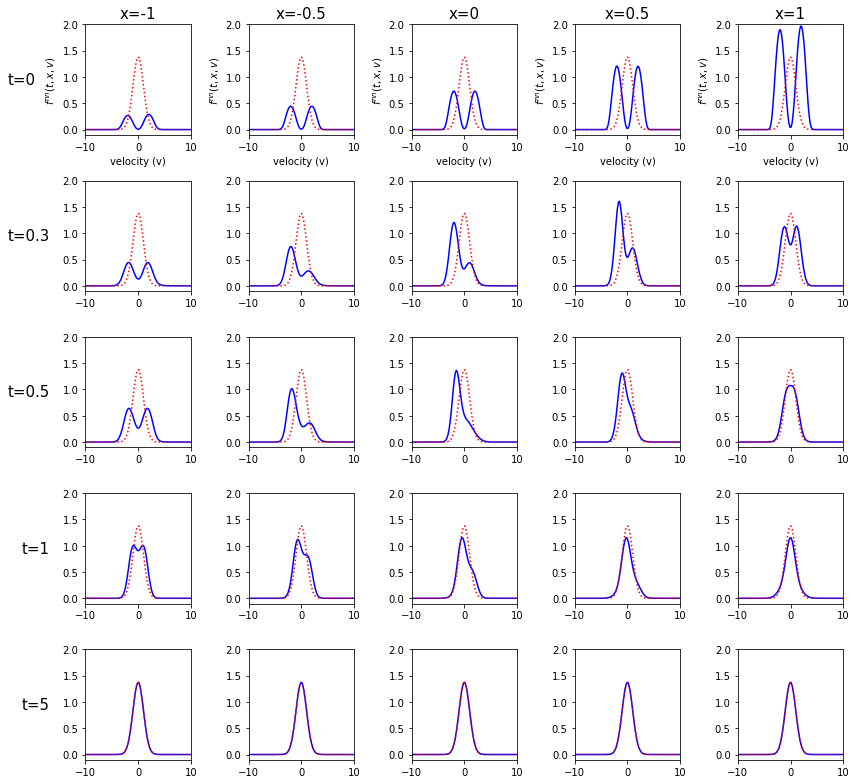}
  \caption{The pointwise values of $f^{nn}(t,x,v;m,w,b)$ as time $t$ varies at each position $x$'s. $x=-1,-0.5,0,0.5,1$ are the points to explain the convergence to the global Maxwellian $\frac{\|f_0(\cdot,\cdot)\|_{L^1_{x,v}}}{|\Omega|}M(v)$. The steady-state (global Maxwellian) is given via the red-dotted lines.}
  \label{fig:vpfp_3}
\end{figure}

Figure \ref{fig:vpfp_3} shows the pointwise values of the approximated neural network solution $f_\varepsilon^{nn}(t,x,v;m,w,b)$ as time $t$ varies at each position $x=-1,-0.5,0,0.5$ and $x=1$. Also, Figure \ref{fig:vpfp_4} shows the pointwise values of the neural network solution $E_\varepsilon^{nn}(t,x;m,w,b))$ as time $t$ varies at some positions $x=-1,-0.5,0,0.5$ and $x=1$ in different colors as shown in the legend. The two plots show that the $f^{nn}(t,x,v;m,w,b)$ and $E^{nn}(t,x;m,w,b)$ converge pointwisely to each equilibrium 
$$f_{\varepsilon,\infty}(x,v)=\frac{\|f_0(\cdot,\cdot)\|_{L^1_{x,v}}}{|\Omega|}M(v)\text{ and }E_{\varepsilon,\infty}(x)=0,$$
which is precisely explained in \eqref{equilibrium_VPFP}. We expect that the steady-state of the distribution $f^{nn}(t,x,v;m,w,b)$ to the VPFP system has the same global Maxwellian at each position $x\in[-1,1]$ although the initial condition \eqref{VPFP_initial} has the different ditributions at each position. To confirm this, we denote the global Maxwellian function $f_{\varepsilon,\infty}(x,v)$ via the red-dotted lines in Figure \ref{fig:vpfp_3}. As we expect, Figure \ref{fig:vpfp_3} shows that the distribution functions $f^{nn}(t,x,v;m,w,b)$ converge to the same Maxwellian shape at time $t=5$. The relative $L^2_{x,v}$ error between the global Maxwellian $f_{\varepsilon,\infty}(x,v)$ and the equilibrium of the neural network solution at $t=5$ is $4.7\times{10}^{-3}$. Also, the pointwise values of $E^{nn}(t,x;m,w,b)$ for all positions $x\in[-1,1]$ converge to zero as shown in Figure \ref{fig:vpfp_4}. This result also shows an agreement with the theoretical steady-state \eqref{equilibrium_VPFP}.

\begin{figure}[H]
  \includegraphics[width=0.8\textwidth, draft=false]{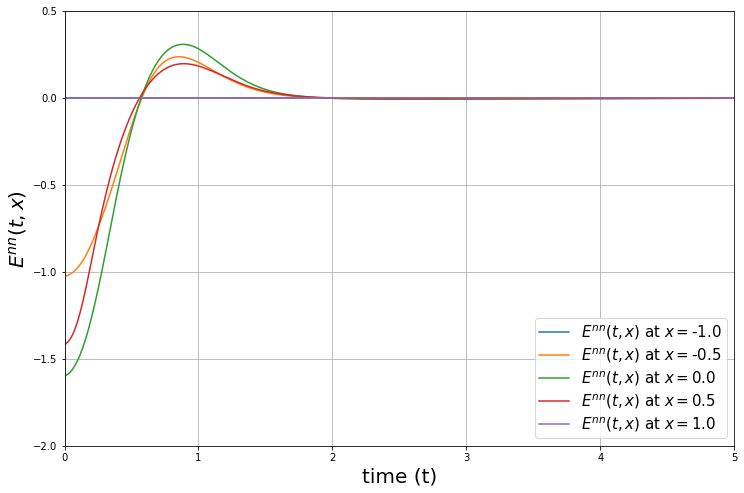}
  \caption{The pointwise values of $E^{nn}(t,x;m,w,b)$ at each position $x$'s over time $t$. The values at each position $x=-1,-0.5,0,0.5,1$ are drawn in different colors as shown in the legend.}
  \label{fig:vpfp_4}
\end{figure}

\section{Part III. On convergence of DNN solutions to an analytic solution to the PNP system and simulation results}\label{sec:part3}
In this section, we provide a DNN solution to the PNP system \eqref{PNP}. This section also consists of three subsections, similarly to Part II (Section \ref{sec:part2}). First, we propose the loss functions for the PNP system. Second, we prove the convergence of a DNN solution to an analytic solution to the PNP system in two steps. Finally, we show the simulation results of the DNN solutions to the PNP system by comparing the time-asymptotic behaviors, the macroscopic quantities, and the steady-state of the PNP system which is defined in Section \ref{subsubsec:VPFP_equilibrium_quantities}.
\subsection{Loss functions for the PNP system}
\label{subsec:part3_loss}
We need to define loss functions for the PNP system: $Loss^{pnp}_{GE}$ for the PNP system \eqref{PNP}$_1$ and \eqref{PNP}$_3$, $Loss^{pnp}_{IC}$ for the initial condition \eqref{PNP}$_2$ and \eqref{PNP}$_4$ and $Loss^{pnp}_{BC}$ for the boundary condition \eqref{no-flux} and \eqref{PNP}$_5$. Note that we use the superscript $Loss^{pnp}$ for all loss functions to the PNP system to distinguish it from the superscript $Loss^{vpfp}$ used for the loss functions to the VPFP system in Section \ref{subsec:part2_loss}.

First, we define loss functions for the governing equation as
\begin{multline}\label{pnp_loss_ge1}
  Loss^{pnp}_{GE^{(1)}}(\rho^{nn})
  \\\eqdef \int_{(0,T)}dt\int_{(-1,1)}dx |\partial_t \rho^{nn}(t,x;m,w,b) - \partial_{xx} \rho^{nn}(t,x;m,w,b)\\
  + \partial_x (\rho^{nn}(t,x;m,w,b)E^{nn}(t,x;m,w,b))|^2,
\end{multline}
and
\begin{multline}\label{pnp_loss_ge2}
	Loss^{pnp}_{GE^{(2)}}(\rho^{nn})
	\\\eqdef \int_{(0,T)}dt\int_{(-1,1)}dx \left|\partial_x E^{nn}(t,x;m,w,b)-\rho^{nn}(t,x;m,w,b)+1\right|^2
	\\+ \int_{(0,T)}dt\int_{(-1,1)}dx \left|\partial_t\left(\partial_x E^{nn}(t,x;m,w,b)-\rho^{nn}(t,x;m,w,b)+1\right)\right|^2
	\\+ \int_{(0,T)}dt\int_{(-1,1)}dx \left|\partial_x\left(\partial_x E^{nn}(t,x;m,w,b)-\rho^{nn}(t,x;m,w,b)+1\right)\right|^2.
\end{multline}
Note that this loss function $Loss^{pnp}_{GE^{(2)}}(\rho^{nn})$ is not just the $L^2$ error with respect to $t$ and $x$. We add the $t$-derivative and the $x$-derivative of the error to the original $L^2$ error as shown in the definition \eqref{pnp_loss_ge2}. We need these two terms to prove the convergence of the neural network solution to the analytic solution in Theorem \ref{thm:pnp_converge} in the following section. Then we define $Loss^{pnp}_{GE}$ as
$$
Loss^{pnp}_{GE}(\rho^{nn}) \eqdef Loss^{pnp}_{GE^{(1)}} + Loss^{pnp}_{GE^{(2)}}.
$$
We now define the loss function for the initial condition
\begin{equation}\label{pnp_loss_ic1}
Loss^{pnp}_{IC^{(1)}}(\rho^{nn}) \eqdef \int_{(-1,1)}dx\left|\rho^{nn}(0,x;m,w,b)-\rho_0(x)\right|^2	
\end{equation}
and
\begin{multline}\label{pnp_loss_ic2}
Loss^{pnp}_{IC^{(2)}}(\rho^{nn})
\\\eqdef \int_{(-1,1)}dx\left|E^{nn}(0,x;m,w,b)-\left(\int_{-1}^x\rho_0(y)dy-(x+1)\right)\right|^2.	
\end{multline}
Then, we define $Loss^{pnp}_{IC}$ as 
$$
Loss^{pnp}_{IC}(\rho^{nn}) \eqdef Loss^{pnp}_{IC^{(1)}} + Loss^{pnp}_{IC^{(2)}}
$$
The loss function for the \textit{Neumann} boundary condition for $\rho(t,x)$ is defined as follows:
\begin{equation}\label{pnp_loss_bc1}
Loss^{pnp}_{BC^{(1)}}(\rho^{nn}) \eqdef \int_{(0,T)}dt\int_{x\in\partial[-1,1]}dx\left|\partial_x\rho^{nn}(t,x;m,w,b)\right|^2.
\end{equation}
We defined the loss function for the \textit{Dirichlet} boundary condition for $E(t,x)$
\begin{multline}\label{pnp_loss_bc2}
	Loss^{pnp}_{BC^{(2)}}(\rho^{nn}) \eqdef \int_{(0,T)}dt\int_{x\in\partial[-1,1]}dx\left(|E^{nn}(t,x;m,w,b)|^2\right)
	\\+ \int_{(0,T)}dt\int_{x\in\partial[-1,1]}dx\left(|\partial_tE^{nn}(t,x;m,w,b)|^2\right).
\end{multline}
Note that we add the error of $\partial_tE^{nn}(t,x;m,w,b)$ to the original $L^2$ error as shown in the definition \eqref{pnp_loss_bc2}. This is also for the proof in Theorem \ref{thm:pnp_converge} in the following section. Then, we define the total loss for the boundary conditions as
$$
Loss^{pnp}_{BC}(\rho^{nn}) \eqdef Loss^{pnp}_{BC^{(1)}} + Loss^{pnp}_{BC^{(2)}}.
$$
Finally, we define the total loss as
\begin{equation}\label{pnp_loss_total}
  Loss^{pnp}_{Total}(\rho^{nn}) \eqdef Loss^{pnp}_{GE} + Loss^{pnp}_{IC} + Loss^{pnp}_{BC}.
\end{equation}
Note that we compute these loss functions via approximating the integration by the Riemann sum on the grid points similarly to Section \ref{subsec:part2_loss}.

\subsection{On convergence of DNN solutions to an analytic solution to the PNP system}\label{subsec:part3_theory}
This section shows the convergence of the DNN solutions to an analytic solution to the PNP system \eqref{PNP} in two steps, similarly to Section \ref{subsec:part3_theory}. First, we prove that there exists a sequence of neural network parameters such that the total loss function $Loss^{pnp}_{Total}$ converges to 0. We then show that the corresponding sequence of DNN solutions converges to an analytic solution if we minimize the total loss function $Loss^{pnp}_{Total}$. Throughout the section, we assume that the existence and the uniqueness of solutions for the PNP system \eqref{PNP} with the no-flux boundary condition \eqref{no-flux} are a priori given.

We introduce our first main theorem similarly to that of Theorem \ref{thm:vpfp_loss_0} which shows the existence of a sequence of neural network parameters that makes the total loss function converge to zero if the $\widehat{C}^{(1,2)}([0,T]\times[-1,1])$ solution to the PNP system exists:
\begin{theorem}[Theorem 3.4 of \cite{MR4116803}]\label{thm:pnp_loss_0}Assume that the solution $\rho$ to \eqref{PNP} with \eqref{no-flux} which belongs to $\widehat{C}^{(1,2)}([0,T]\times[-1,1])$, and the activation function $\bar{\sigma}(x)\in C^{(1,2)}([0,T]\times [-1,1])$ is non-polynomial. Then, there exists $\{m_{[j]}, w_{[j]}, b_{[j]}\}_{j=1}^\infty$ such that a sequence of the DNN solutions $\rho^{nn}$ of Theorem \ref{VPFP_global} with $m_{[j]}$ nodes, denoted by $$\{\rho_j(t,x,v) = \rho^{nn}(t,x,v;m_{[j]}, w_{[j]}, b_{[j]})\}_{j=1}^{\infty}$$ satisfies\footnote{Each of $m_{[j]}, w_{[j]}, b_{[j]}$ represents the matrix of the numbers corresponding to $\rho_j$ for each $j=1,2,...,\infty$. The matrices $m_{[j]}, w_{[j]}, b_{[j]}$ consist of the element represented as $m_{[j],ik}^{(l)}, w_{[j],ik}^{(l)}, b_{[j],ik}^{(l)}$, respectively.}
\begin{equation}\label{pnp_loss_0}
Loss^{pnp}_{Total}(\rho_j)\rightarrow 0\text{ as }j\rightarrow\infty.
\end{equation}
\end{theorem}

Now we introduce our second main theorem, which shows that the sequence of DNN solutions converges to an analytic solution to the PNP system in a suitable function space when we minimize the total loss function $Loss^{pnp}_{Total}$. We also refer to Remark \ref{remark:VPFP_theory_simulation}, which explains how these main theorems are related to our Deep Learning algorithm.
\begin{theorem}\label{thm:pnp_converge} Assume that $\rho$ is a solution to \eqref{PNP} with \eqref{no-flux} which belongs to $\widehat{C}^{(1,2)}([0,T]\times[-1,1])$. Then, the Deep Neural Network solution $\rho^{nn}(t,x;m,w,b)$ satisfies that
\begin{equation}
\left\| \rho^{nn}(\cdot,\cdot;m,w,b)  - \rho\right\|_{L_t^{\infty}([0,T];L_x^{2}([-1,1]))} \leq CLoss^{pnp}_{Total}(\rho^{nn}),
\end{equation}
where $C$ is a positive constant depending only on $T$.
\end{theorem}
\begin{proof}We define the error values of the neural network outputs $\rho^{nn}$ and $E^{nn}$ as the following equations:
$$d_{ge,j}^{(1)}(t,x)\eqdef -\left[\partial_{t}-\partial_{xx}\right]\rho^{nn}+\partial_{x}(\rho^{nn}E^{nn})$$
and
$$ d_{ge,j}^{(2)}(t,x)\eqdef -(\partial_xE^{nn} - \rho^{nn}(t,x)+h(x))$$
for $(t,x)\in[0,T]\times[-1,1]$,
$$d_{bc,j}^{(1)}(t,x)\eqdef -\partial_x\rho^{nn}(t,x)$$
and
$$ d_{bc,j}^{(2)}(t,x)\eqdef -E^{nn}(t,x)$$
for $x=\pm 1$. Then, we consider the following equation on the difference between $\rho$ and $\rho^{nn}$ for each fixed $j$ on the compact set of $t,x$ only as 
\begin{equation}\label{PNP_diff_ge_1}
\left[\partial_{t}-\partial_{xx}\right]\{\rho-\rho^{nn}\}+\partial_{x}(\rho E - \rho^{nn}E^{nn})=d_{ge,j}^{(1)}(t,x)
\end{equation}
Then we derive the inequality below by multiplying $2(\rho-\rho^{nn})$ onto \eqref{PNP_diff_ge_1} and integrating it over $[-1,1]$ as
\begin{multline}\label{PNP_energy_eq}
\int_{-1}^{1}\frac{\partial}{\partial t}(\rho-\rho^{nn})^{2}(t,x)dx -2\int_{-1}^1(\rho - \rho^{nn})\partial_{xx}(\rho - \rho^{nn})dx
\\ = -2\int_{-1}^1(\rho - \rho^{nn})\partial_x(\rho E - \rho^{nn}E^{nn})dx + 2\langle d_{ge,j}^{(1)},(\rho-\rho^{nn})\rangle, 
\end{multline}
where $\langle\cdot,\cdot\rangle$ denotes the standard inner product on ${L^2_x([-1,1])}$. On the left-hand side of \eqref{PNP_energy_eq}, we note that
$$\int_{-1}^{1}\frac{\partial}{\partial t}(\rho-\rho^{nn})^{2}(t,x)dx=\frac{\partial}{\partial t}\left\|\rho-\rho^{nn}\right\|_{L^2_x([-1,1])}^2$$

and
\begin{multline}\int_{-1}^1(\rho - \rho^{nn})\partial_{xx}(\rho - \rho^{nn})dx
\\=\int_{\partial[-1,1]}(\rho - \rho^{nn})d_{bc,j}^{(1)}n_xdS_x-\left\|\partial_x(\rho - \rho^{nn})\right\|_{L^2_x([-1,1])}^2
\\\leq \frac{1}{2}\left\|\rho-\rho^{nn}\right\|_{L^2_x(\partial[-1,1])}^2 + \frac{1}{2}\left\|d_{bc,j}^{(1)}\right\|_{L^2_x(\partial[-1,1])}^2 - \left\|\partial_x(\rho - \rho^{nn})\right\|_{L^2_x([-1,1])}^2
\\\leq \frac{1}{2}\left(\left\|\rho-\rho^{nn}\right\|_{L^2_x([-1,1])}^2 + \left\|\partial_x(\rho-\rho^{nn})\right\|_{L^2_x([-1,1])}^2\right) \\+ \frac{1}{2}\left\|d_{bc,j}^{(1)}\right\|_{L^2_x(\partial[-1,1])}^2 - \left\|\partial_x(\rho - \rho^{nn})\right\|_{L^2_x([-1,1])}^2
\\\leq \frac{1}{2}\left\|\rho-\rho^{nn}\right\|_{L^2_x([-1,1])}^2 - \frac{1}{2}\left\|\partial_x(\rho-\rho^{nn})\right\|_{L^2_x([-1,1])}^2 + \frac{1}{2}\left\|d_{bc,j}^{(1)}\right\|_{L^2_x(\partial[-1,1])}^2,
\end{multline}
by the trace theorem for $\rho-\rho^{nn}$. Since $\rho$ belongs to $\widehat{C}^{(1,2)}([0,T]\times[-1,1])$ and $\rho^{nn}$ is a continuous function with respect to $x$, it implies that $|(\rho-\rho^{nn})(t,x)|<M_1$ on the compact domain $x\in\partial[-1,1]$ for some positive constant $M_1$. Also for the second term on the right-hand side of \eqref{PNP_energy_eq}, we note that
$$ \langle d_{ge,j}^{(1)},(\rho-\rho^{nn})\rangle \leq \frac{1}{2} (\left\|d_{ge,j}^{(1)}\right\|_{L^2_x([-1,1])}^2 + \left\|\rho-\rho^{nn}\right\|_{L^2_x([-1,1])}^2).$$

Therefore, we reduce \eqref{PNP_energy_eq} to
\begin{multline}\label{PNP_energy_eq_reduce}
\frac{\partial}{\partial t}\left\|\rho-\rho^{nn}\right\|_{L^2_x([-1,1])}^2 \\ -\left\|\rho-\rho^{nn}\right\|_{L^2_x([-1,1])}^2 + \left\|\partial_x(\rho-\rho^{nn})\right\|_{L^2_x([-1,1])}^2 - \left\|d_{bc,j}^{(1)}\right\|_{L^2_x(\partial[-1,1])}^2
\\ \leq -2\int_{-1}^1(\rho - \rho^{nn})\partial_x(\rho E - \rho^{nn}E^{nn})dx + \left\|d_{ge,j}^{(1)}\right\|_{L^2_x([-1,1])}^2 + \left\|\rho-\rho^{nn}\right\|_{L^2_x([-1,1])}^2.
\end{multline}
Also, we reduce the absolute value of the first term on the right hand side of \eqref{PNP_energy_eq_reduce} to
\begin{multline}\label{b_1_b_2}
\left| -2\int_{-1}^1(\rho - \rho^{nn})\partial_x(\rho E - \rho^{nn}E^{nn})dx\right|\\
\leq \underbrace{2\left|\int_{\partial[-1,1]}  (\rho - \rho^{nn})(\rho E - \rho^{nn}E^{nn})n_xdS_x\right|}_{\eqdef B_1(t)}\\
+ \underbrace{2\left|\int_{-1}^1  \partial_x(\rho - \rho^{nn})(\rho E - \rho^{nn}E^{nn})dx\right|}_{\eqdef B_2(t)},
\end{multline}
by the integration by parts with respect to $x$. Note that we have
\begin{multline}\label{PNP_B1}
  B_1(t) = 2\left|\int_{\partial[-1,1]}  (\rho - \rho^{nn})(-\rho^{nn}E^{nn})n_xdS_x\right|
  \\ = 2\left|\int_{\partial[-1,1]}  (\rho - \rho^{nn})\left((\rho-\rho^{nn})-\rho\right)E^{nn}n_xdS_x\right|
  \\ \leq 2\left|\int_{\partial[-1,1]}  (\rho - \rho^{nn})^2d_{bc,j}^{(2)}n_xdS_x\right| + 2\left|\int_{\partial[-1,1]}  (\rho - \rho^{nn})\rho d_{bc,j}^{(2)}n_xdS_x\right|.
\end{multline}
Therefore, by Holder's inequality, trace theorem and the Cauchy-Schwarz inequality with $\epsilon_0$, we can bound $B_1(t)$ as
\begin{multline}
  B_1(t) \leq 2\left\|d_{bc,j}^{(2)}\right\|_{L^\infty_x(\partial[-1,1])}\left\|\rho-\rho^{nn}\right\|_{L^2_x(\partial[-1,1])}^2
  \\+ 2M_1\left\|\rho-\rho^{nn}\right\|_{L^2_x(\partial[-1,1])}\left\|d_{bc,j}^{(2)}\right\|_{L^2_x(\partial[-1,1])}
  \\\leq \frac{1}{3}\left\|\rho-\rho^{nn}\right\|_{L^2_x(\partial[-1,1])}^2 + M_1\left(\epsilon_0\left\|\rho-\rho^{nn}\right\|_{L^2_x(\partial[-1,1])}^2+\frac{1}{\epsilon_0}\left\|d_{bc,j}^{(2)}\right\|_{L^2_x(\partial[-1,1])}^2\right)
  \\ \leq \left(\frac{1}{3}+M_1\epsilon_0\right)\left\|\rho-\rho^{nn}\right\|_{L^2_x(\partial[-1,1])}^2 + \frac{M_1}{\epsilon_0}\left\|d_{bc,j}^{(2)}\right\|_{L^2_x(\partial[-1,1])}^2
  \\ \leq \left(\frac{1}{3}+M_1\epsilon_0\right)\left( \left\|\rho-\rho^{nn}\right\|_{L^2_x([-1,1])}^2+\left\|\partial_x(\rho-\rho^{nn})\right\|_{L^2_x([-1,1])}^2 \right)
  \\ + \frac{M_1}{\epsilon_0}\left\|d_{bc,j}^{(2)}\right\|_{L^2_x(\partial[-1,1])}^2,
\end{multline}
since we can reduce $Loss^{pnp}_{BC^{(2)}}$ defined in \eqref{pnp_loss_bc2} sufficently small so we can bound $\left\|d_{bc,j}^{(2)}\right\|_{L^\infty_x(\partial[-1,1])}$ for all time $t\in[0,T]$ using the Sobolev embedding theorem as follows:
\begin{equation}\begin{split}\label{smallness_dbc2}
&\left\|\left\|d_{bc,j}^{(2)}\right\|_{L^\infty_x(\partial[-1,1])}\right\|_{L^\infty_t([0,T])}
\\&\qquad\leq\left\|d_{bc,j}^{(2)}(t,x=-1)\right\|_{L^\infty_t([0,T])}+\left\|d_{bc,j}^{(2)}(t,x=1)\right\|_{L^\infty_t([0,T])}
\\&\qquad\leq C_0\left(\left\|d_{bc,j}^{(2)}(t,x=-1)\right\|_{H^1_t([0,T])}+\left\|d_{bc,j}^{(2)}(t,x=1)\right\|_{H^1_t([0,T])}\right)
\\&\qquad\leq \sqrt{2}C_0\left( \left\|d_{bc,j}^{(2)}(t,x=-1)\right\|_{H^1_t([0,T])}^2+\left\|d_{bc,j}^{(2)}(t,x=1)\right\|_{H^1_t([0,T])}^2 \right)
\\&\qquad= \sqrt{2}C_0Loss^{pnp}_{BC^{(2)}}<\frac{1}{6},
\end{split}\end{equation}
for some positive constant $C_0$. We have some positive constant $M_1$ satisfying $|\rho(t,x)|<M_1$ on $x\in\partial[-1,1]$, since $\rho$ belongs to $\widehat{C}^{(1,2)}([0,T]\times[-1,1])$. By choosing a sufficently small $\epsilon_0$ so that $\frac{1}{3}+M_1\epsilon_0<\frac{1}{2}$, then we can bound $B_1(t)$ as
\begin{multline}\label{PNP_B1_final}
	B_1(t)\leq \frac{1}{2}\left\|\rho-\rho^{nn}\right\|_{L^2_x([-1,1])}^2+\frac{1}{2}\left\|\partial_x(\rho-\rho^{nn})\right\|_{L^2_x([-1,1])}^2
	\\+ C_1\left\|d_{bc,j}^{(2)}\right\|_{L^2_x(\partial[-1,1])}^2,
\end{multline}
for some positive constant $C_1$.

To bound the second term $B_2(t)$ on the right-hand side of the inequality \eqref{b_1_b_2}, we have
\begin{multline}\label{PNP_B2}
  B_2(t) \leq 2\left|\int_{-1}^1  \partial_x(\rho - \rho^{nn})\rho(E-E^{nn})dx\right|
  \\+ 2\left|\int_{-1}^1  \partial_x(\rho - \rho^{nn})E^{nn}(\rho-\rho^{nn})dx\right|
  \\ \leq \underbrace{\left\|\rho\right\|_{L^\infty_x([-1,1])}\left( \frac{1}{\epsilon_1}\left\|E-E^{nn}\right\|_{L^\infty_x([-1,1])}^2+\epsilon_1\left\|\partial_x(\rho-\rho^{nn})\right\|_{L^2_x([-1,1])}^2 \right)}_{\eqdef B_{2,1}(t)}
  \\+ \underbrace{2\left\|E^{nn}\right\|_{L^\infty_x([-1,1])}\int_{-1}^1|(\rho-\rho^{nn})\partial_x(\rho-\rho^{nn})|dx}_{\eqdef B_{2,2}(t)},
\end{multline}
by the Cauchy-Schwarz inequality with $\epsilon_1$. We consider the equation on the difference between $E$ and $E^{nn}$ for each fixed $j$ on the compact set of $t,x$ only as
\begin{multline}
\left|(E-E^{nn})(t,x)\right|
\\= \left|\int_{-1}^x\left((\rho-\rho^{nn})(t,x')+d_{ge,j}^{(2)}(t,x')\right) dx' + (E-E^{nn})(t,x=-1)\right|
\\ \leq \int_{-1}^x|\rho-\rho^{nn}|(t,x')dx'+\int_{-1}^x|d_{ge,j}^{(2)}(t,x')|dx' + \left|E^{nn}(t,x=-1)\right|. 
\end{multline}
Thus, we have
\begin{multline}\label{diff_E_bdd}
\left\|E-E^{nn}\right\|_{L_x^\infty}\leq \left\|\rho-\rho^{nn}\right\|_{L^1([-1,1])} + \left\| d_{ge,j}^{(2)}\right\|_{L^1([-1,1])} + \left|E^{nn}(t,x=-1)\right|
\\ \leq \sqrt{2}\left\|\rho-\rho^{nn}\right\|_{L^2_x([-1,1])} + \sqrt{2}\left\| d_{ge,j}^{(2)}\right\|_{L^2_x([-1,1])}
\\+ \left(\left|E^{nn}(t,x=-1)\right|+\left|E^{nn}(t,x=1)\right|\right)
\\ \leq \sqrt{2}\left\|\rho-\rho^{nn}\right\|_{L^2_x([-1,1])} + \sqrt{2}\left\| d_{ge,j}^{(2)}\right\|_{L^2_x([-1,1])} + \left\|d_{bc,j}^{(2)}\right\|_{L^1(\partial[-1,1])}
\\ \leq \sqrt{2}\left(\left\|\rho-\rho^{nn}\right\|_{L^2_x([-1,1])} + \left\| d_{ge,j}^{(2)}\right\|_{L^2_x([-1,1])} + \left\|d_{bc,j}^{(2)}\right\|_{L^2_x(\partial[-1,1])}\right),
\end{multline}
by Holder's inequality. Then the first term $B_{2,1}(t)$ on the right-hand side of the inequality \eqref{PNP_B2} is bounded as
\begin{multline}\label{PNP_B2_1}
B_{2,1}(t)
\\\leq M_2\Bigg( \frac{2}{\epsilon_1}\left(\left\|\rho-\rho^{nn}\right\|_{L^2_x([-1,1])} + \left\| d_{ge,j}^{(2)}\right\|_{L^2_x([-1,1])} + \left\|d_{bc,j}^{(2)}\right\|_{L^2_x(\partial[-1,1])}\right)^2
\\ +\epsilon_1\left\|\partial_x(\rho-\rho^{nn})\right\|_{L^2_x([-1,1])}^2 \Bigg)
\\\leq M_2\Bigg( \frac{6}{\epsilon_1}\left(\left\|\rho-\rho^{nn}\right\|_{L^2_x([-1,1])}^2 + \left\| d_{ge,j}^{(2)}\right\|_{L^2_x([-1,1])}^2+\left\|d_{bc,j}^{(2)}\right\|_{L^2_x(\partial[-1,1])}^2\right)
\\ +\epsilon_1\left\|\partial_x(\rho-\rho^{nn})\right\|_{L^2_x([-1,1])}^2 \Bigg),
\end{multline}
with $M_2$ satisfies $|\rho(t,x)|<M_2$ on $x\in[-1,1]$ since the solution $\rho$ belongs to $\widehat{C}^{(1,2)}([0,T]\times[-1,1])$.

Now, we consider the second term $B_{2,2}(t)$ on the right-hand side of the inequality \eqref{PNP_B2}. Since we reduce $Loss^{pnp}_{GE^{(2)}}$ sufficiently small, we can bound the $\left\|d_{ge,j}^{(2)}\right\|_{L^2_x([-1,1])}$ for all time $t\in[0,T]$ as
\begin{multline}\label{smallness_dge2}
\left\|\left\|d_{ge,j}^{(2)}\right\|_{L^2_x([-1,1])}\right\|_{L^\infty_x([0,T])} \leq \left\|\left\|d_{ge,j}^{(2)}\right\|_{L^\infty_x([-1,1])}\right\|_{L^\infty_x([0,T])}
\\ \leq \left\|d_{ge,j}^{(2)}\right\|_{L^\infty_x([0,T]\times[-1,1])} \leq C_2\left\|d_{ge,j}^{(2)}\right\|_{H^1_{t,x}([0,T]\times[-1,1])} = C_2 {Loss^{pnp}_{GE^{(2)}}}^\frac{1}{2}\leq M_4,
\end{multline}
for some positive constant $M_4$ by the Sobolev embedding theorem. Therfore, we can bound $\left\|E-E^{nn}\right\|_{L^\infty_x([-1,1])}$ in the inequality \eqref{diff_E_bdd} as
\begin{multline}
	\left\|E-E^{nn}\right\|_{L^\infty_x([-1,1])}
	\\\leq \sqrt{2}\left\|\rho-\rho^{nn}\right\|_{L^2_x([-1,1])} + \sqrt{2}\underbrace{\left\| d_{ge,j}^{(2)}\right\|_{L^2_x([-1,1])}}_{\leq M_4} + \sqrt{2}\underbrace{\left\|d_{bc,j}^{(2)}\right\|_{L^\infty_x(\partial[-1,1])}}_{\leq\frac{1}{6}}
	\\\leq \sqrt{2}\left\|\rho-\rho^{nn}\right\|_{L^2_x([-1,1])} + M_4 + \frac{1}{6}
\end{multline}
using the inequalities \eqref{smallness_dbc2} and \eqref{smallness_dge2}. Using this bound, we can bound the second term $B_{2,2}(t)$ on the right-hand side of the inequality \eqref{PNP_B2} as
\begin{align*}
B_{2,2}(t)&\leq 2\left\|\rho-\rho^{nn}\right\|_{L^2_x}\left\|\partial_x(\rho-\rho^{nn})\right\|_{L^2_x}\left( \left\|E-E^{nn}\right\|_{L^\infty_x([-1,1])}+ \left\|-E\right\|_{L^\infty_x([-1,1])} \right)
\\&\leq 2\left\|\rho-\rho^{nn}\right\|_{L^2_x}\left\|\partial_x(\rho-\rho^{nn})\right\|_{L^2_x}
\\&\times\left(\sqrt{2}\left\|\rho-\rho^{nn}\right\|_{L^2_x([-1,1])} + \sqrt{2}M_4 + \frac{\sqrt{2}}{6} + M_3\right).
\end{align*}
Defining $M_5\eqdef\sqrt{2}M_4 + \frac{\sqrt{2}}{6} + M_3$ yields that
\begin{multline}\label{PNP_B3}
B_{2,2}(t)
\\\leq 2\sqrt{2}\left\|\rho-\rho^{nn}\right\|_{L^2_x}^2\left\|\partial_x(\rho-\rho^{nn})\right\|_{L^2_x} + 2M_5\left\|\rho-\rho^{nn}\right\|_{L^2_x}\left\|\partial_x(\rho-\rho^{nn})\right\|_{L^2_x}
\\\leq \frac{\sqrt{2}}{\epsilon_2}\left\|\rho-\rho^{nn}\right\|_{L^2_x}^4+\sqrt{2}\epsilon_2\left\|\partial_x(\rho-\rho^{nn})\right\|_{L^2_x}^2 
\\+ M_5\left(\frac{1}{\epsilon_2}\left\|\rho-\rho^{nn}\right\|_{L^2_x}^2 + \epsilon_2\left\|\partial_x(\rho-\rho^{nn})\right\|_{L^2_x}^2\right)
\end{multline}
for any positive constant $\epsilon_2$ and some positive constant $M_3$ satisfying $|E(t,x)|<M_3$ on $x\in[-1,1]$ since $E$ is the continuous function with respect to $x$. Therefore, we can bound $B_{2,2}(t)$ as
\begin{multline}
B_{2,2}(t)\leq \frac{M_5}{\epsilon_2} \left\|\rho-\rho^{nn}\right\|_{L^2_x([-1,1])}^2 + \frac{\sqrt{2}}{\epsilon_2}\left\|\rho-\rho^{nn}\right\|_{L^2_x([-1,1])}^4
\\+ (M_5+\sqrt{2})\epsilon_2\left\|\partial_x(\rho-\rho^{nn})\right\|_{L^2_x([-1,1])}^2.
\end{multline}

Therefore, we can bound $B_2(t)\leq B_{2,1}(t)+B_{2,2}(t)$ in the inequality \eqref{PNP_B3} as
\begin{multline}\label{PNP_B2_4}
  B_2(t)
  \\\leq M_2\Bigg( \frac{6}{\epsilon_1}\left(\left\|\rho-\rho^{nn}\right\|_{L^2_x}^2 + \left\| d_{ge,j}^{(2)}\right\|_{L^2_x}^2+\left\|d_{bc,j}^{(2)}\right\|_{L^2_x(\partial[-1,1])}^2\right) +\epsilon_1\left\|\partial_x(\rho-\rho^{nn})\right\|_{L^2_x}^2 \Bigg)
  \\ + \frac{M_5}{\epsilon_2} \left\|\rho-\rho^{nn}\right\|_{L^2_x}^2 + \frac{\sqrt{2}}{\epsilon_2}\left\|\rho-\rho^{nn}\right\|_{L^2_x}^4 + (M_5+\sqrt{2})\epsilon_2\left\|\partial_x(\rho-\rho^{nn})\right\|_{L^2_x}^2
  \\ \leq \left(\frac{6M_2}{\epsilon_1}+\frac{M_5}{\epsilon_2}\right)\left\|\rho-\rho^{nn}\right\|_{L^2_x}^2 +  \frac{\sqrt{2}}{\epsilon_2} \left\|\rho-\rho^{nn}\right\|_{L^2_x}^4 + \frac{6M_2}{\epsilon_1}\left\|d_{ge,j}^{(2)}\right\|_{L^2_x}^2
  \\+ \frac{6M_2}{\epsilon_1}\left\|d_{bc,j}^{(2)}\right\|_{L^2_x(\partial[-1,1])}^2 + \left(M_2\epsilon_1+(M_5+\sqrt{2})\epsilon_2\right)\left\|\partial_x(\rho-\rho^{nn})\right\|_{L^2_x}^2.
\end{multline}
by Holder's inequality.
By choosing sufficently small $\epsilon_1$ and $\epsilon_2$ so that $$M_2\epsilon_1+(M_5+\sqrt{2})\epsilon_2<\frac{1}{2},$$ then $B_2(t)$ is bounded as
\begin{multline}\label{PNP_B2_final}
	B_2(t)\leq C_3\left(\left\|\rho-\rho^{nn}\right\|_{L^2_x([-1,1])}^2 + \left\|\rho-\rho^{nn}\right\|_{L^2_x([-1,1])}^4\right)
	\\ + C_4\left(\left\|d_{ge,j}^{(2)}\right\|_{L^2_x([-1,1])}^2 + \left\|d_{bc,j}^{(2)}\right\|_{L^2_x(\partial[-1,1])}^2\right) + \frac{1}{2}\left\|\partial_x(\rho-\rho^{nn})\right\|_{L^2_x([-1,1])}^2
\end{multline}
for some positive constant $C_3$ and $C_4$. Therefore, by using the boundedness of $B_1(t)$ in \eqref{PNP_B1_final} and $B_2(t)$ in \eqref{PNP_B2_final}, the inequality \eqref{b_1_b_2} yields that
\begin{multline}
\left| -2\int_{-1}^1(\rho - \rho^{nn})\partial_x(\rho E - \rho^{nn}E^{nn})dx\right| \leq B_1(t) + B_2(t) 
\\\leq \frac{1}{2}\left\|\rho-\rho^{nn}\right\|_{L^2_x([-1,1])}^2+\frac{1}{2}\left\|\partial_x(\rho-\rho^{nn})\right\|_{L^2_x([-1,1])}^2 + C_1\left\|d_{bc,j}^{(2)}\right\|_{L^2_x(\partial[-1,1])}^2
\\+C_3\left(\left\|\rho-\rho^{nn}\right\|_{L^2_x([-1,1])}^2 + \left\|\rho-\rho^{nn}\right\|_{L^2_x([-1,1])}^4\right)
\\ + C_4\left(\left\|d_{ge,j}^{(2)}\right\|_{L^2_x([-1,1])}^2 + \left\|d_{bc,j}^{(2)}\right\|_{L^2_x(\partial[-1,1])}^2\right) + \frac{1}{2}\left\|\partial_x(\rho-\rho^{nn})\right\|_{L^2_x([-1,1])}^2
\\\leq \left(C_3+\frac{1}{2}\right)\left\|\rho-\rho^{nn}\right\|_{L^2_x([-1,1])}^2 + C_3 \left\|\rho-\rho^{nn}\right\|_{L^2_x([-1,1])}^4 + \left\|\partial_x(\rho-\rho^{nn})\right\|_{L^2_x([-1,1])}^2
\\+ \left(C_1+C_4\right)\left\|d_{bc,j}^{(2)}\right\|_{L^2_x(\partial[-1,1])}^2 + C_4\left\|d_{ge,j}^{(2)}\right\|_{L^2_x([-1,1])}^2.
\end{multline}

Therefore, we reduce \eqref{PNP_energy_eq_reduce} to
\begin{multline}\label{PNP_energy_eq_reduce_2}
\frac{\partial}{\partial t}\overbrace{\left\|\rho-\rho^{nn}\right\|_{L^2_x([-1,1])}^2}^{Y(t)\eqdef }  -\left\|\rho-\rho^{nn}\right\|_{L^2_x([-1,1])}^2 + \left\|\partial_x(\rho-\rho^{nn})\right\|_{L^2_x([-1,1])}^2
\\- \left\|d_{bc,j}^{(1)}\right\|_{L^2_x(\partial[-1,1])}^2
\\ \leq \left(C_3+\frac{1}{2}\right)\left\|\rho-\rho^{nn}\right\|_{L^2_x}^2 + C_3 \left\|\rho-\rho^{nn}\right\|_{L^2_x}^4 + \left\|\partial_x(\rho-\rho^{nn})\right\|_{L^2_x}^2
\\+ (C_1+C_4)\left\|d_{bc,j}^{(2)}\right\|_{L^2_x(\partial[-1,1])}^2 + C_4\left\|d_{ge,j}^{(2)}\right\|_{L^2_x}^2 + \left\|d_{ge,j}^{(1)}\right\|_{L^2_x}^2+ \left\|\rho-\rho^{nn}\right\|_{L^2_x}^2.
\end{multline}
If we define the constant $C_5\eqdef C_2+\frac{5}{2}$, then we can rewrite inequality \eqref{PNP_energy_eq_reduce_2} as follows
\begin{multline}\label{PNP_energy_eq_reduce_3}
Y'(t) - C_5Y(t) \leq C_3Y(t)^2
\\+ C_6\underbrace{\left(\left\| d_{ge,j}^{(1)}\right\|_{L^2_x([-1,1])}^2+\left\| d_{ge,j}^{(2)}\right\|_{L^2_x([-1,1])}^2+\left\|d_{bc,j}^{(1)}\right\|_{L^2_x(\partial[-1,1])}^2+\left\|d_{bc,j}^{(2)}\right\|_{L^2_x(\partial[-1,1])}^2\right)}_{\eqdef L(t)},
\end{multline}
with some positive constant $C_6$. Multiplying \eqref{PNP_energy_eq_reduce_3} by $e^{-C_5t}$ and integrating it over $[0,t]$ for $t<T$, we have
\begin{multline}
e^{-C_5t}Y(t)-Y(0)\leq C_3\int_0^{t}e^{-C_5s}Y(s)^2ds + C_6\int_{0}^{t}e^{-C_5s}L(s)ds
\\ \leq C_3\int_0^{t}Y(s)^2ds + C_6\int_{0}^{t}L(s)ds.
\end{multline}
Therefore, we have the inequality
\begin{equation}
Y(t)\leq e^{C_5t}\left( Y(0)+C_6\int_{0}^{t}L(s)ds \right) + C_3e^{C_5t}\int_0^{t}Y(s)^2ds.
\end{equation}
By \cite[Theorem 25]{MR2016992}, we have
\begin{multline}\label{PNP_energy_eq_reduce_4}
Y(t)\\
\leq e^{C_5t}\left( Y(0)+C_6\int_{0}^{t}L(s)ds \right)\left[ 1-C_3\int_0^{t}e^{2C_5s'}\left( Y(0)+C_6\int_{0}^{s'}L(s)ds \right)ds' \right]^{-1}
\end{multline}
for $0\leq t \leq \beta$, where
\begin{equation}\label{beta_condition}
	\beta=\sup\bigg\{ t\in[0,T]\Big|C_3\int_0^{t}e^{2C_5s'}\underbrace{\left( Y(0)+C_6\int_{0}^{s'}L(s)ds \right)}_{\eqdef h(s')}ds'<1 \bigg\}.
\end{equation}
Note that $Y(0)=Loss^{pnp}_{IC}(\rho^{nn})$ and
$$Y(0)+C_6\int_{0}^{t}L(s)ds \leq Y(0)+C_6\int_{0}^{T}L(s)ds=Loss^{pnp}_{Total}(\rho^{nn}).$$
Therefore, the value of $h(s')$ in \eqref{beta_condition} is bounded as $Loss^{pnp}_{Total}(\rho^{nn})$ which can be sufficiently small. Therefore, we can choose $t=T$ in \eqref{PNP_energy_eq_reduce_4} to obtain that
\begin{multline}
\left\|\rho-\rho^{nn}\right\|_{L^2_x([-1,1])}^2
\\\leq e^{C_5T}Loss^{pnp}_{Total}(\rho^{nn})\left[ 1-C_3\int_0^{T}e^{2C_5s'}\left( Y(0)+C_6\int_{0}^{s'}L(s)ds \right)ds' \right]^{-1}
\\\leq C Loss^{pnp}_{Total}(\rho^{nn})
\end{multline}
for some positive constant $C$. Therefore, this completes the proof of Theorem.
\end{proof}

\subsection{Neural Network simulations}\label{subsec:part3_simulation}In this section, we provide numerical simulations for the solutions $\rho^{nn}(t,x;m,w,b)$ and $E^{nn}(t,x;m,w,b)$ to the PNP system \eqref{PNP}. We set the initial condition of $\rho^{nn}(t,x;m,w,b)$ as follows:
\begin{equation}\label{PNP_initial}
\rho_0(0,x)=\int_\R dv f_0(x,v) = 8e^{x-1}.
\end{equation}
Note that we set the initial condition \eqref{PNP_initial} which satisfies $\rho(0,x)=\rho_0(x)=\int_\R dv f_0(x,v)$ to compare the convergence on the solutions of the VPFP system with the soluitons of the PNP system in Part IV. We also set the background charge $h(x)$ as constant to satisfy $\int_\Omega\rho(0,x)-h(x)=0$. The details of our Deep Learning algorithm are explained in Section \ref{subsec:dnn_architerture}, Section \ref{subsec:grid}, and the summary of Algorithm \ref{algorithm}.

Figure \ref{fig:PNP_1} shows the total density \eqref{mass rho conservation} and the free energy \eqref{PNP_FE} of the PNP system. As shown in the left plot in Figure \ref{fig:PNP_1}, the total density $\text{Mass}_\rho(t)$ of the neural network solution $\rho^{nn}(t,x;m,w,b)$ is conserved. It is well-matched to the theoretical result as in \eqref{mass rho conservation}, which is an important property in the PNP system with the no-flux boundary condition \eqref{no-flux}. The right plot in Figure \ref{fig:PNP_1} shows the free energy of the neural network solution with the steady-state value via the red-dotted line. We compute the steady-state value of the free energy using the steady state of the PNP solution, $\rho_\infty(x)$ and $E_\infty(x)$, in \eqref{equilibrium_PNP}. We observe that the free energy is non-increasing as shown in the right plot in Figure \ref{fig:PNP_1}. It verifies that the neural network solutions, $\rho^{nn}(t,x;m,w,b)$ and $E^{nn}(t,x;m,w,b)$, of the PNP system satisfy the dissipation law of the free energy as explained in \eqref{FE_dissipation}.

Also, we expect that the free energy decreases exponentially to the steady-state based on Theorem 1.2 in \cite{biler2000long}. In Figure \ref{fig:PNP_2}, the plot shows the time evolution of the free energy of the neural network solution in a log-linear scale. We compute the decreasing rate of the free energy with the difference between the free energy $\text{FE}_\rho(t)$ and the steady-state of the free energy $\text{FE}_{\rho,\infty}$ at $t=1$. We also denote the algebraic rates and the geometric rates in log-linear scale. We can observe that the decreasing rate of the free energy is almost simlilar to the geometric rate $Ce^{-11.7}t$, which is a linear function in Figure \ref{fig:PNP_2}, except for a small error near the final time.

\begin{figure}[H]
  \includegraphics[width=\textwidth, draft=false]{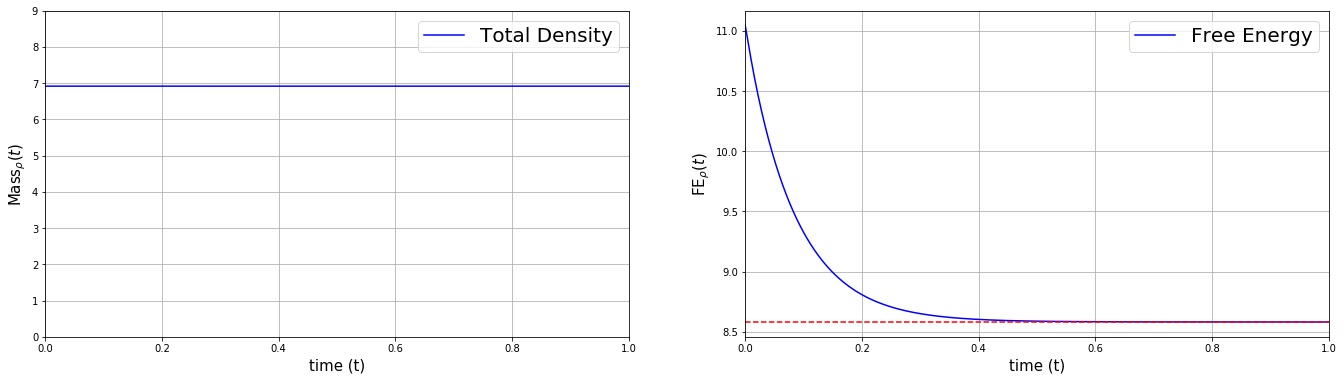}
  \caption{The time-asymptotic behaviors of the total density $\text{Mass}_\rho(t)$ and the free energy $\text{FE}_\rho(t)$ of the PNP system. In the second plot, the steady-state value of the free energy is indicated in the red-dotted line. Note that the free energy is monotonically decreasing.}
  \label{fig:PNP_1}
\end{figure}

\begin{figure}[H]
  \includegraphics[width=0.7\textwidth, draft=false]{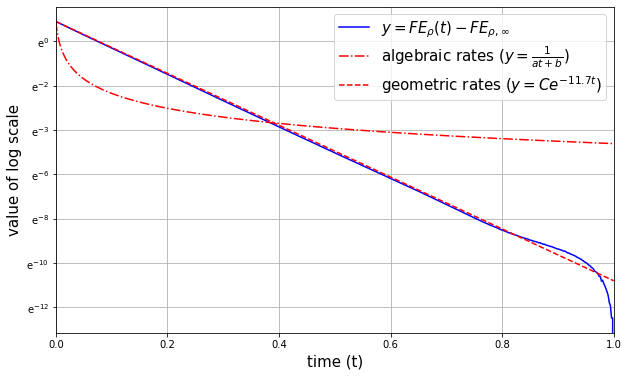}
  \caption{The time-asymptotic behaviors of the difference between the free energy $\text{FE}_\rho(t)$ and the steady-state of the free energy $\text{FE}_{\rho,\infty}$ in log-linear scale. We consider the numerical solution at $t=1$ as the steady state of the free energy. Note that this plot verifies the exponential decay of the approximated free energy.}
  \label{fig:PNP_2}
\end{figure}

Figure \ref{fig:PNP_3} shows the poinstwise values of each approximated neural network solution $\rho^{nn}(t,x;m,w,b)$ and $E^{nn}(t,x;m,w,b)$ in different colors as time $t$ varies at each position $x$ from -1 to 1. We expect that the neural network solutions $\rho^{nn}(t,x;m,w,b)$ and $E^{nn}(t,x;m,w,b)$ converge pointwisely to each equilibrium 
$$\rho_\infty(x)=C_{pnp}
\text{ and
}E_\infty(x)=0,$$
which is precisely explained in \eqref{equilibrium_PNP}. As shown in the first plot in Figure \ref{fig:PNP_3}, the neural network solution $\rho^{nn}(t,x;m,w,b)$ converges to constant for all $x$. It is well consistent to the theoretical supports provided in \eqref{equilibrium_PNP}. Also, the second plot in Figure \ref{fig:PNP_3} shows that the neural network solution $E^{nn}(t,x;m,w,b)$ converges to zero for all $x\in[-1,1]$ as $t$ increases. This simulation result also well matches the expected steady state of the PNP system as in \eqref{equilibrium_PNP}.

\begin{figure}[H]
  \includegraphics[width=\textwidth, draft=false]{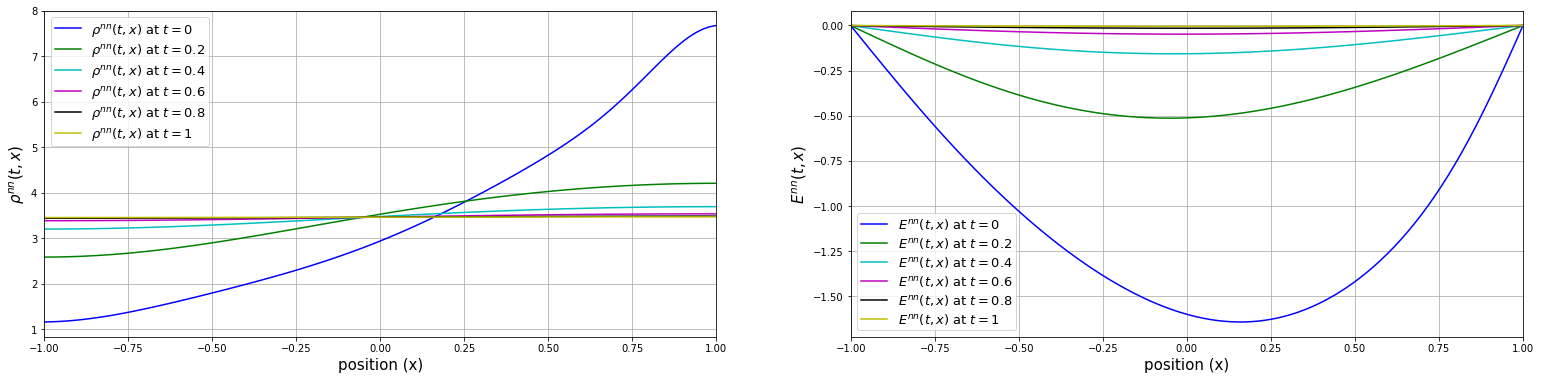}
  \caption{The pointwise values of $\rho^{nn}(t,x;m,w,b)$ and $E^{nn}(t,x;m,w,b)$ for each position $x$ as time $t$ varies. The values at each time $t=0,0.2,0.4,0.6,0.8,1$ are drawn in different colors as shown in the legend.}
  \label{fig:PNP_3}
\end{figure}

\section{Part IV. On the simulation results of the diffusion limit from the VPFP system to the PNP system}\label{sec:part4}
In this section, we provide the trend of the diffusion limit from the VPFP system to the PNP system using the simulation results of our Deep Neural Network approach. We consider the convergence of the VPFP solutions to the PNP solution, as summarized in Part I (Section \ref{sec:part1}). We expect that the neural network solutions of the VPFP system and the PNP system have the trend of diffusion limit as explained in the equations \eqref{converge_f} and \eqref{converge_E}. To observe the trend of the convergence, we compare the neural network solutions to the VPFP system with the Knudsen numbers $\varepsilon=1,0.5,0.2,0.1,0.05$ and the corresponding neural network solutions to the PNP system. The methods of how to train the neural network solutions to the VPFP system and the PNP system are precisely described in Section \ref{sec:methodology}. Also, the results of the numerical simulations are given in Part II (Section \ref{sec:part2}) for the VPFP system and Part III (Section \ref{sec:part3}) for the PNP system, respectively.

As we have introduced in Section \ref{subsec:grid_reuse} on the simulation methodology, we use the `Grid Reuse' method to capture the VPFP with the small Knudsen number $\varepsilon$. When the `Grid Reuse' strategy is not used, the neural network solutions to the VPFP system could not approximate well at the early part of the time grid (about 0.0$\sim$0.2 time grids) as the Knudsen number $\varepsilon$ is smaller. This means that the VPFP system with the small Knudsen number is hard to be approximated at the early time grid using Deep Learning. Therefore, the `Grid Reuse' method is essential to observe the diffusion limit from the neural network solution of the VPFP system to the neural network solution of the PNP system.

We define the total loss function \eqref{VPFP_loss_total} in the sense of the Mean Square Error (MSE). In this section, we use the Root Mean Square Error (RMSE) as the loss function for the VPFP system. These two cases show almost similar results, but we choose the RMSE loss function that offers better results. We use 50 reused grid points. The details of our Deep Learning algorithm are explained in Section \ref{subsec:dnn_architerture} and the summary of the Deep Learning Algorithm \ref{algorithm}.

\subsection{Neural Network simulations} In this section, we present the results of the numerical simulations for the diffusion limit from the VPFP system to the PNP system. We set the initial condtion \eqref{VPFP_initial} for the VPFP system and the initial condition \eqref{PNP_initial} for the PNP system. It is worth noting that we do not change the number of the grid points for the VPFP system with any Knudsen numbers, i.e., we anlayze the diffusion limit in the sense of the Asymptotic-Preserving (AP) scheme. Instead, we use the `Grid Reuse' method for all neural network solutions.

\begin{figure}[H]
  \includegraphics[width=0.9\textwidth, draft=false]{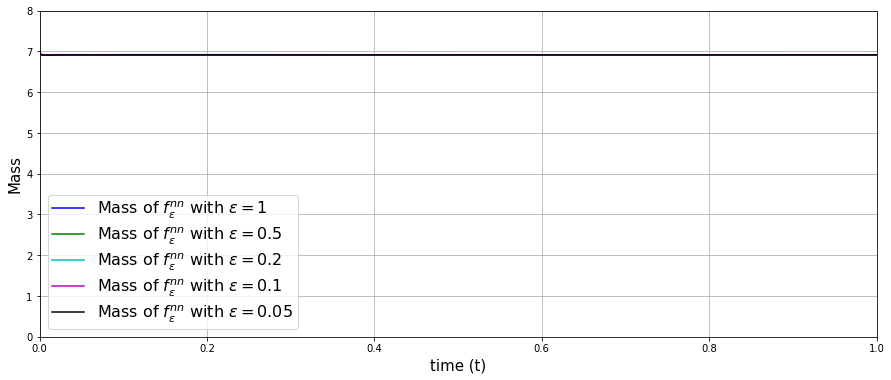}
  \caption{The time-asymptotic behaviors of the total mass of $f_\varepsilon^{nn}(t,x,v;m,w,b)$ ($L^1_{x,v}$ norm of $f_\varepsilon^{nn}$) as $\varepsilon$ varies. The values with each Knudsen number $\varepsilon$ are drawn in different colors as shown in the legend. It is notable that the total mass of the distribution $f_\varepsilon^{nn}$ is conserved over time although the Knudsen number $\varepsilon$ varies.} 
  \label{fig:diffusive_1}
\end{figure}

\begin{figure}[H]
  \includegraphics[width=0.9\textwidth, draft=false]{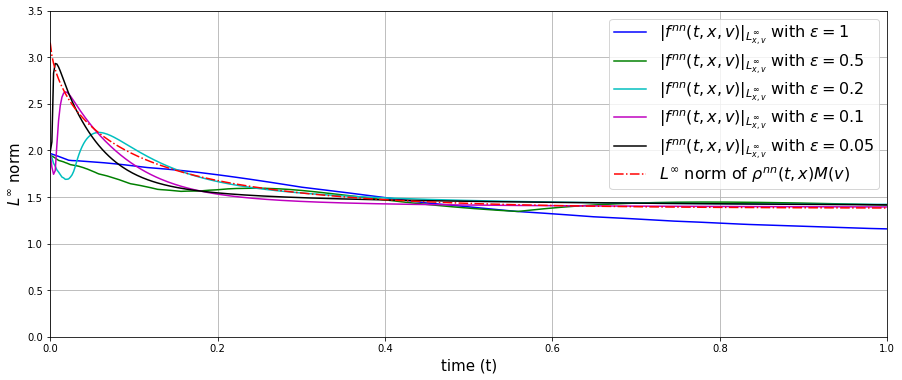}
  \caption{The time-asymptotic behavior of the $L^\infty_{x,v}$ norm of $f_\varepsilon^{nn}(t,x,v;m,w,b)$ and $\rho^{nn}(t,x;m,w,b)M(v)$ over time $t$ as the Knudsen number $\varepsilon$ varies. Each value is drawn in different colors as shown in the legend.}
  \label{fig:diffusive_2}
\end{figure}

Figure \ref{fig:diffusive_1} indicates the total mass $\int_{\Omega\times V}f_\varepsilon^{nn}dvdx$ of the VPFP system with the different Knudsen numbers $\varepsilon=1, 0.5, 0.2, 0.1$ and $0.05$ in different colors as shown in the legend. As shown in Figure \ref{fig:diffusive_1}, all five graphs overlap so that they appear as one graph. This is because that all the five cases conserve the total mass over time. This plot implies that the neural network solutions for all 5 cases well approximate the solution of the VPFP system.

Figure \ref{fig:diffusive_2} shows the $L^\infty$ norm of the solution $f_\varepsilon^{nn}(t,x,v;m,w,b)$ to the VPFP system with the different Knudsen numbers $\varepsilon=1, 0.5, 0.2, 0.1$ and $0.05$ in different colors as shown in the legend. Also, we plot the $L^\infty$ norm of the neural network solution $\rho^{nn}(t,x;m,w,b)M(v)$ via the red-dotted-line in Figure \ref{fig:diffusive_2}. We can observe that the $L^\infty$ norm of the solution $f_\varepsilon^{nn}(t,x,v;m,w,b)$ converges pointwisely to the $L^\infty$ norm of the $\rho^{nn}(t,x;m,w,b)M(v)$ as the Knudsen number $\varepsilon$ becomes close to zero. This gives more information than the theoretical result of convergence as explained in \eqref{converge_f}.

\begin{figure}[H]
  \includegraphics[width=0.9\textwidth, draft=false]{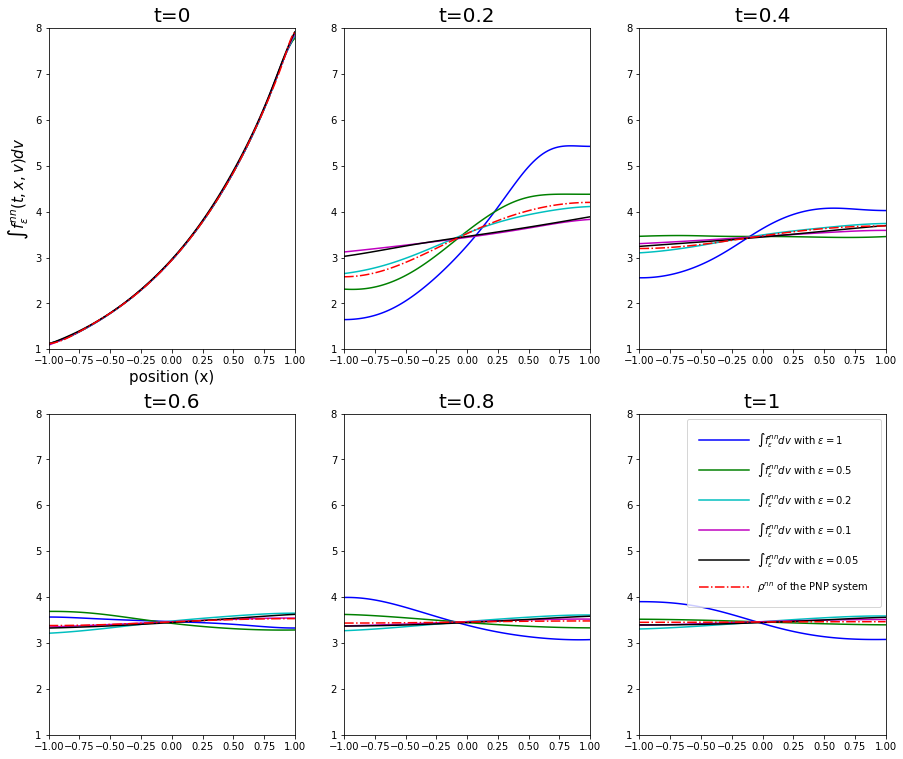}
  \caption{The pointwise values of $\int  f_\varepsilon^{nn}(t,x,v;m,w,b)dv$ to the VPFP system and $\rho^{nn}(t,x;m,w,b)$ to the PNP system for each position $x$ at time $t=0,0.2,0.4,0.6,0.8,1$ as $\varepsilon$ varies. The values with each Knudsen number $\varepsilon$ are drawn in different colors as shown in the legend.}
  \label{fig:diffusive_3}
\end{figure}

The graphs in Figure \ref{fig:diffusive_3} and in Figure \ref{fig:diffusive_4} show the pointwise values of the solutions as time $t$ varies at each $x$'s. Figure \ref{fig:diffusive_3} shows the pointwise values of $\int f_\varepsilon^{nn}(t,x,v;m,w,b)dv$ as the Knudsen number $\varepsilon$ varies in different colors as shown in the legend. We also plot the pointwise values of $\rho^{nn}(t,x;m,w,b)$ via the red-dotted lines. The first plot in Figure \ref{fig:diffusive_3} shows that the initial condition \eqref{VPFP_initial} for the neural network solution of the VPFP system is consistent with the initial condition \eqref{PNP_initial} for the neural network solution of the PNP system, since we set the initial conditions to satisfy the relation $\int_\R f_0(t,x,v)dv=\rho_0(t,x)$. It is remarkable that the neural network solutions to the VPFP system with the different Knudsen numbers well approximate the initial condtion and the same for the solution to the PNP system. Also, we expect that the integration of neural network solution $\int f_\varepsilon(t,x,v)dv$ converges to the density $\rho(t,x;m,w,b)$ which is consistent to the convergence of $f_\varepsilon$ to $\rho(t,x)M(v)$ as explained in \eqref{converge_f}. As shown in the six plots in Figure \ref{fig:diffusive_3}, the pointwise values of the $\int f_\varepsilon(t,x,v)dv$ to the VPFP system converge to the solution $\rho(t,x;m,w,b)$ for all time $t\in[0,1]$ as the small Knudsen number $\varepsilon$ becomes small. It is consistent to the theoretical supports in \eqref{converge_f}.

\begin{figure}[H]
  \includegraphics[width=0.9\textwidth, draft=false]{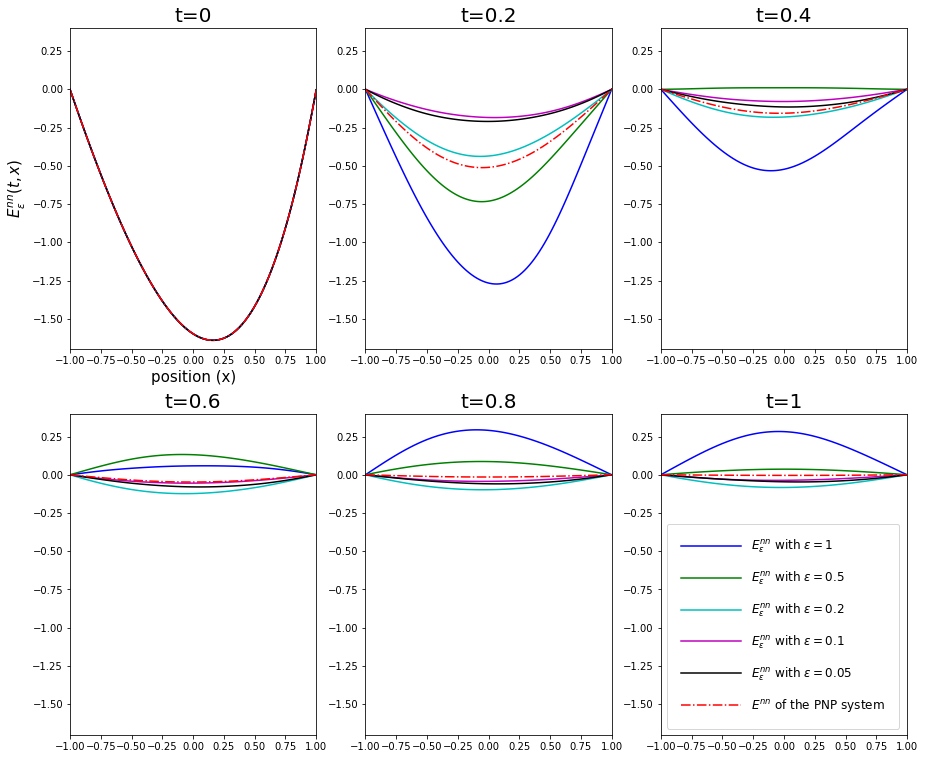}
  \caption{The pointwise values of $E^{nn}_\varepsilon(t,x;m,w,b)$ to the VPFP system and $E^{nn}(t,x;m,w,b)$ to the PNP system for each position $x$ at time $t=0,0.2,0.4,0.6,0.8,1$ as $\varepsilon$ varies. The values with each Knudsen number $\varepsilon$ are drawn in different colors as shown in the legend.}
  \label{fig:diffusive_4}
\end{figure}

Figure \ref{fig:diffusive_4} shows the pointwise values of the electric force $E_\varepsilon^{nn}(t,x;m,w,b)$ with the differnt Knudsen numbers $\varepsilon=1,0.5,0.2,0.1$ and $\varepsilon=0.05$ in different colors as time $t$ varies at each $x$'s. Also, we plot the electric force of the PNP system $E^{nn}(t,x;m,w,b)$ in the red-dotted lines. We remark that the first plot in Figure \ref{fig:diffusive_4} shows the same pointwise values of the electric force to the VPFP system with different Knudsen numbers and to the PNP system. It means that the initial conditions of the neural network solution of the electric force $E^{nn}$ to the VPFP system and the PNP system are well approximated. Also, we observe that the six plots in Figure \ref{fig:diffusive_4} show the solution $E^{nn}_\varepsilon(t,x;m,w,b)$ of the VPFP system converges to the solution $E^{nn}(t,x;m,w,b)$ of the PNP system as the Knudsen number $\varepsilon$ goes to zero. It agrees with the theoretical result as explained in \eqref{converge_E}.

\begin{figure}[H]
  \includegraphics[width=0.8\textwidth, draft=false]{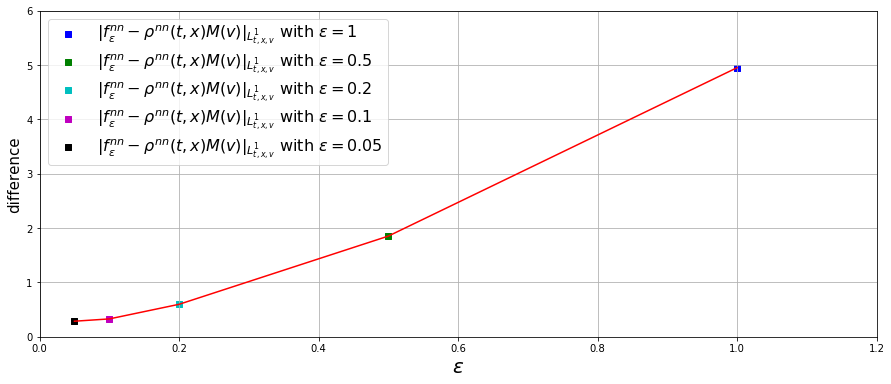}
  \caption{The values of $L_{t,x,v}^1$ norm of the difference between $f^{nn}_\varepsilon(t,x,v;m,w,b)$ and $\rho^{nn}(t,x;m,w,b)M(v)$ as $\varepsilon$ varies.}
  \label{fig:diffusive_5}
\end{figure}

Finally, Figure \ref{fig:diffusive_5} shows that the $L_{t,x,v}^1$ norm of the difference between the distribution $f^{nn}_\varepsilon(t,x,v;m,w,b)$ and the solution $\rho^{nn}(t,x;m,w,b)M(v)$ as $\varepsilon$ varies. We note that Figure \ref{fig:diffusive_5} shows the convergence of $f_\varepsilon$ to $\rho(t,x)M(v)$ more quantitatively than the plots in the previous figures. As we expected in $\eqref{converge_f}$, the graph shows that the $L_{t,x,v}^1$ norm of the difference between $f^{nn}_\varepsilon$ and $\rho^{nn}M$ becomes smaller as the Knudsen number $\varepsilon$ tends to zero.

\section{Conclusion}\label{sec:conclusion}
In this paper, we establish the commutation of the diagram of diffusion limit in Figure \ref{fig:diagram}. This also implies the reduction of the VPFP system with the specular boundary condition to the PNP system with the no-flux boundary condition as the Knudsen number $\varepsilon$ tends to zero. To this end, we have introduced the Deep Neural Network (DNN) solutions to the VPFP system and the PNP system using the Deep Learning algorithm. We use the two neural networks to approximate the VPFP system and the PNP system, coupled with the Poisson equation. Also, we propose appropriate loss functions for training, including the loss function for the initial conditions and the boundary conditions to each system: the VPFP system in Part II and the PNP system in Part III. We also provide the theoretical supports on which the approximated DNN solutions converge to analytic solutions of each system as the proposed total loss function tends to zero. We also provide the numerical simulations on the DNN solutions of each system, which support the theoretical predictions on the asymptotic behaviors of each system. These include the steady-states for the solutions and the physical quantities such as the total mass, the kinetic energy, the entropy, the electric energy, and the free energy.

Finally, using these DNN solutions of the two systems, we observe the trend of the diffusion limit in Part IV. We analyze our DNN solutions based on the theory shown in Part I. We use the newly devised technique `Grid Reuse' method adapted to the Deep Learning algorithm. This technique makes it possible to approximate the solution of the VPFP system with the Knudsen number $\varepsilon$ in the range between 0.05 and 1. An improved method to approximate the VPFP system will be needed to make it work with a smaller value of the Knudsen number $\varepsilon$ than 0.05, and this is one of the future works in the direction.

One of the difficulties that the Deep Learning approach experiences as a PDE solver is on its rate of convergence and stability. Compared to the traditional numerical schemes which have a great amount of well-known studies on each method's performance, the Deep Learning method still has some difficulties in dealing with such things due to the optimization issues. However, it is also true that the Deep Learning is a new approach with many advantages as the Deep Learning method is a mesh-free method. The Deep Learning algorithm that we introduce in this paper does not require itself to have the mesh generation and instead it re-samples the grid points for each domain in every epoch. We expect that our work can be applied to arbitrary domains in higher dimensional kinetic equations.

\section*{Acknowledgement} 
J. Y. Lee and H. J. Hwang were supported by the Basic Science
Research Program through the National Research Foundation of Korea \newline(NRF- 2017R1E1A1A03070105 and NRF-2019R1A5A1028324). 

J. W. Jang was supported by the CRC 1060 \textit{The Mathematics of Emergent Effects} of the
University of Bonn funded through the German Science Foundation (DFG).

\bibliographystyle{amsplaindoi} 
\bibliography{bibliography}

\end{document}